\documentclass{article}
\title{A derived category analogue of the Nakai--Moishezon criterion}
\renewcommand\footnotemark{}
\thanks{The second author is partially supported by the National Science Foundation under Award No. 2402087.}
\author{Daigo Ito and Noah Olander}
\date{}

\usepackage{stix} 
\usepackage[english]{babel} 
\usepackage{amsmath,amsfonts,amsthm,amscd}
\usepackage{ascmac}
\usepackage{appendix}
\usepackage{bm}
\usepackage{cases}
\usepackage{changepage}
\usepackage{enumitem}
\usepackage{etoolbox}
\usepackage{float}
\usepackage{geometry}
\usepackage{graphicx}
\usepackage[utf8]{inputenc}
\usepackage [autostyle, english = american]{csquotes}
\usepackage{scalerel}
\usepackage{ytableau}
\usepackage{tikz-cd} 
\usetikzlibrary{patterns}
\usepackage{tocloft}
\usepackage{setspace}
\usepackage[color,all]{xy}
\usepackage[cal=euler]{mathalfa}
\usepackage{url}
\usepackage{CJKutf8}
\makeatletter
\newcommand{\address}[1]{\gdef\@address{#1}}
\newcommand{\email}[1]{\gdef\@email{\url{#1}}}
\newcommand{\website}[1]{\gdef\@website{\url{#1}}}
\newcommand{\@endstuff}{\par\vspace{\baselineskip}\noindent\small
\begin{tabular}{@{}l}\scshape{Daigo Ito} \\ \scshape\@address\\\textrm{E-mail address:} \@email \\\textrm{Website:} \@website\end{tabular}\\
\begin{tabular}{@{}l}\scshape{Noah Olander} \\ \scshape\@address\\\textrm{E-mail address:} \url{nolander@berkeley.edu} \\\textrm{Website:} \url{https://noaholander.github.io/}\end{tabular}}
\AtEndDocument{\@endstuff}
\makeatother
\address{Department of Mathematics, University of California, Berkeley, Evans Hall, CA 94720-3840}
\email{daigoi@berkeley.edu}
\website{https://daigoi.github.io/}

\usepackage[alphabetic,backrefs]{amsrefs}
\usepackage{hyperref}
\hypersetup{colorlinks,linkcolor=blue,citecolor=red}

\DeclareMathOperator{\pic}{Pic}

\newcommand {\bb}{\mathbb}

\newcommand{\ecal}{\mathscr}
\renewcommand {\epsilon}{\varepsilon}

\newcommand {\bra}[1]{\langle{#1}\rangle}
\renewcommand {\l}{\left}

\renewcommand {\r}{\right}

\newcommand*{\DashedArrow}[1][]{\mathbin{\tikz [baseline=-0.25ex,-latex, dashed,#1] \draw [#1] (0pt,0.5ex) -- (1.3em,0.5ex);}}
\newcommand {\ratmap}{\DashedArrow[->,densely dashed    ]} 

\newcommand {\emp}{\emptyset}

\newcommand {\inj}{\hookrightarrow}
\newcommand {\inv}{^{-1}}

\newcommand {\surj}{\twoheadrightarrow}
\newcommand {\tens}{\otimes}

\newsavebox{\pullbacks}
\sbox\pullbacks{%
\begin{tikzpicture}%
\draw (0,0) -- (1ex,0ex);%
\draw (1ex,0ex) -- (1ex,1ex);%
\end{tikzpicture}}

\let \choose \relax
\newcommand{\choose}[2]{\genfrac{(}{)}{0pt}{}{#1}{#2}}

\let \sf \relax 
\newcommand{\sf}{\mathsf}

\theoremstyle{plain}
\newtheorem{theorems}{Theorem}[section] 
\newtheorem{claims}{Claim}[theorems]
\newtheorem{conjectures}[theorems]{Conjecture}
\newtheorem{corollaries}[theorems]{Corollary}
\newtheorem{lemmas}[theorems]{Lemma} 
\newtheorem{props}[theorems]{Proposition}

\newtheorem{penmdef}[claims]{Definition}

\theoremstyle{definition}
\newtheorem{constructions}[theorems]{Construction}
\newtheorem{definitions}[theorems]{Definition} 
\newtheorem{axioms}[theorems]{Axiom} 
\newtheorem{examples}[theorems]{Example}     
\newtheorem{notations}[theorems]{Notation}         

\newtheorem{question}[theorems]{Question}
\newtheorem{obss}[theorems]{Observation}
\newtheorem{penmlem}[claims]{Lemma}
\newtheorem{penmthm}[claims]{Theorem}
\newtheorem{penmcor}[claims]{Corollary}
\newtheorem{penmeg}[claims]{Example} 
\newtheorem{inclaims}[claims]{Claim}

\theoremstyle{remark}
\newtheorem{remarks}[theorems]{Remark}

\newtheorem{penmrem}[claims]{Remark}

\makeatletter
\newtheoremstyle{indented}
  {1pt}
  {1pt}
  {\addtolength{\@totalleftmargin}{1.5em}
   \addtolength{\linewidth}{-1.5em}
   \parshape 1 1.5em \linewidth}
  {}
  {\bfseries}
  {.}
  {.5em}
  {}
\makeatother

\theoremstyle{indented}
\newtheorem{pinddef}[claims]{Definition} 
\newtheorem{pindlem}[claims]{Lemma}
\newtheorem{pindthm}[claims]{Theorem}
\newtheorem{pindcor}[claims]{Corollary}
\newtheorem{pindeg}[claims]{Example} 
\newtheorem{pindrem}[claims]{Remark} 
\newtheorem{pindq}[claims]{Question} 
\newtheorem{pindc}[claims]{Conjecture} 
\newtheorem{pindclaim}[claims]{Claim}

\newenvironment{theorem}
{
	\pushQED{\qed}\begin{theorems}}
	{\popQED\end{theorems}}

\newenvironment{prop}
{
	\pushQED{\qed}\begin{props}}
	{\popQED\end{props}}

\newenvironment{notation}
{
	\pushQED{\qed}\begin{notations}}
	{\popQED\end{notations}}

\newenvironment{conjecture}
{
	\pushQED{\qed}\begin{conjectures}}
	{\popQED\end{conjectures}}

\newenvironment{corollary}
{
	\pushQED{\qed}\begin{corollaries}}
	{\popQED\end{corollaries}}

\newenvironment{definition}
{
	\pushQED{\qed}\begin{definitions}}
	{\popQED\end{definitions}}

\newenvironment{lemma}
{
	\pushQED{\qed}\begin{lemmas}}
	{\popQED\end{lemmas}}

\newenvironment{remark}
{
	\pushQED{\qed}\begin{remarks}}
	{\popQED\end{remarks}}
	
\newenvironment{example}
{
	\pushQED{\qed}\begin{examples}}
	{\popQED\end{examples}}

\newenvironment{construction}
{
	\pushQED{\qed}\begin{constructions}}
	{\popQED\end{constructions}}

\MakeOuterQuote{"}

\DeclareMathOperator{\tensgen}{\tens\text{-}Gen}
\begin{document}
\maketitle

\begin{abstract}
    We give a complete characterization of the line bundles on a proper variety whose tensor powers generate the derived category, answering a 2010 question of Chris Brav. The condition is analogous to the Nakai--Moishezon criterion and  can be stated purely in terms of classical notions of positivity of line bundles. There is also a generalization which works for all Noetherian schemes. We use our criterion to prove basic properties of such line bundles and provide non-trivial examples of them. As an application, we give new examples of varieties which can be reconstructed from their derived categories in the sense of the Bondal--Orlov Reconstruction Theorem \cite{bondal_orlov_2001}.
\end{abstract}

\tableofcontents
\section{Introduction}

In modern algebraic geometry, a distinction is made between an embedded variety $X \subset \mathbb{P}^n$ and an abstract variety obtained by glueing affine varieties. 
The latter notion is useful because many varieties which occur in practice, for example as moduli spaces, come with no distinguished projective embedding, and sometimes no projective embedding at all. Given an abstract variety $X$, it is often important to determine (1) whether $X$ embeds into projective space, and (2) if so, how. The main tool for answering both of these questions is the theory of ample line bundles, those line bundles $\ecal{L}$ on a variety $X$ such that for some integer $n > 0$ the global sections of $\ecal{L}^{\otimes n}$ define a closed immersion $X \hookrightarrow \mathbb{P}^N$ of $X$ into projective space. 

The pullback of an ample line bundle under a birational morphism of proper varieties is almost never ample, and for this reason it is helpful in birational geometry to consider the larger class of big line bundles. A line bundle $\ecal{L}$ on a proper variety $X$ is \emph{big} if for some integer $n > 0$, the global sections of $\ecal{L}^{\otimes n}$ define a generically injective rational map to projective space $X \dashrightarrow \mathbb{P}^N$. 
In contrast to the class of ample line bundles, the class of big line bundles is stable under pullbacks along birational morphisms, and it is usually much easier to check bigness of a line bundle than ampleness.
Indeed, an exceptionally useful test for ampleness is the Nakai--Moishezon criterion, which in one formulation says a line bundle $\ecal{L}$ on a proper variety $X$ is ample if and only if the restriction of $\ecal{L}$ to every closed subvariety $Z \subset X$ (including $X$ itself) is big.

 Ampleness of line bundles plays a key role in many developments in the study of derived categories of algebraic varieties. One reason for this is that if $\ecal{L}$ or $\ecal{L}^{-1}$ is an ample line bundle on a variety $X$, then by \cite[Theorem 4]{Orlov_dimension}:

($\ast$) \emph{There is an integer $n > 0$ such that the quasi-coherent derived category of $X$ is compactly generated by the line bundles $\ecal{O}_X, \ecal{L}, \ecal{L}^{\otimes 2}, \cdots , \ecal{L}^{\otimes n}$.} 

\noindent Compact generators are used in many foundational results in the subject, including those on Grothendieck duality, K-theory of schemes, representing functors by Fourier--Mukai kernels, studying semi-orthogonal decompositions of varieties, and much more; see for example the papers  \cite{thomason2007higher}, \cite{Neeman1992}, \cite{Neeman-Grothendieck},  \cite{bondal2002generators}, \cite{orlov2003derived}, \cite{toen2007moduli}, \cite{ben2010integral}, \cite{kuznetsov2011base}. Another foundational result is the Bondal--Orlov Reconstruction Theorem \cite{bondal_orlov_2001}, which states that if $X$ is a smooth projective variety with canonical bundle $\omega_X$ and either $\omega_X$ or $\omega_X^{-1}$ is ample, then $X$ can be reconstructed from its category of perfect complexes. In recent works of the first named author and Matsui \cite{ito2024new}, \cite{ito2025polarizations}, it is pointed out that this holds more generally when $X$ is a Gorenstein proper variety and the condition ($\ast$) holds for $\ecal{L} = \omega_X$. 

It is therefore natural to ask which line bundles $\ecal{L}$ on a scheme $X$ satisfy condition ($\ast$). Indeed this is done in a 2010 Mathoverflow question of Chris Brav \cite{ChrisBrav}, and this is the question we consider in this paper. In 2017, the Mathoverflow user Libli observes that if $k$ is a field and $b : X \to \mathbb{P}^2_k$ is the blowup in a $k$-point with exceptional divisor $E$, then the line bundle $b^*\ecal{O}(1) \otimes \ecal{O}(E)$ also satisfies ($\ast$) \cite{Blowupofpointexample}. This is to our knowledge the first example of a line bundle $\ecal{L}$ on a proper variety $X$ such that neither $\ecal{L}$ nor $\ecal{L}^{-1}$ is ample but ($\ast$) holds. Following \cite{ito2025polarizations}, we shall say a line bundle $\ecal{L}$ on a quasi-compact and quasi-separated scheme $X$ is \emph{$\otimes$-generating} if ($\ast$) holds. See \ref{subsection-tensorample}.

In this paper, we give a complete answer to Brav's question for any Noetherian scheme $X$. When $X$ is a proper scheme over a field, our criterion has the pleasing feature that it can be stated in terms of the familiar notion of bigness. 

\begin{theorem}
\label{thm-mainproperversion}
    Let $X$ be a proper scheme over a field. Let $\ecal{L}$ be a line bundle on $X$. Then $\ecal{L}$ is $\otimes$-generating if and only if for every closed subvariety $Z \subset X$, either $\ecal{L}|_{Z}$ or $\ecal{L}^{-1}|_{Z}$ is big. 
\end{theorem}

Here closed subvariety means integral closed subscheme, and it is necessary to include the irreducible components, i.e., $X$ itself if $X$ is a variety, see Example \ref{example-strictlynefnotample}. This bears a striking similarity to the Nakai--Moishezon criterion, which says that $\ecal{L}$ as above is ample if and only if for every closed subvariety $Z \subset X$, the restriction $\ecal{L}|_Z$ is big. 

The general criterion is as follows.

\begin{theorem}
\label{thm-mainintroversion}
    Let $X$ be a Noetherian scheme. Let $\ecal{L}$ be a line bundle on $X$. Then $\ecal{L}$ is $\otimes$-generating if and only if for every integral closed subscheme $Z \subset X$, there is an integer $n \in \mathbb{Z}$ and a global section $s \in \Gamma(Z, \ecal{L}^{\otimes n}|_{Z})$ such that the open subscheme $Z_s = \{s \neq 0\} \subset Z$ defined by $s$ is non-empty and affine. 
\end{theorem}

These Theorems are stated together as Theorem \ref{ref-theoremmain} and proved in Section \ref{section-mainresult}. They show in particular that a $\otimes$-generating line bundle on a Noetherian scheme $X$ always has some tensor power which has many global sections. This is not immediately obvious from the definition of $\otimes$-generation. Indeed, ($\ast$) implies that there are many morphisms in the derived category $\ecal{L}^{\otimes m}[i] \to \ecal{L}^{\otimes n}[j]$ for \emph{some} integers $i, j, m, n$, but it is unclear whether we can take $i = j$. Our methods show that if $X$ is a smooth proper geometrically connected variety of positive dimension over a field, and $S$ is a collection of line bundles on $X$ such that $\operatorname{Hom}_{\ecal{O}_X}(\ecal{L}, \ecal{M}) = 0$ for distinct elements $\ecal{L}, \ecal{M}$ (but with no assumption on higher Ext groups), then $S$ does not generate $D_{\operatorname{QCoh}}(X)$.

As a consequence of our characterization, we can prove a number of properties of line bundles on Noetherian schemes which are $\otimes$-generating that we are unable to prove directly from the definition.

\begin{corollary}
\label{corollary-consequences}
    Let $\ecal{L}$ be a line bundle on a Noetherian scheme $X$. Let $ 0 \neq n \in \mathbb{Z}$. Let $f : Y \to X$ be a finite surjective morphism of schemes.
    \begin{enumerate}
        \item $\ecal{L}$ is $\otimes$-generating if and only if $\ecal{L}^{\otimes n}$ is $\otimes$-generating.
        \item $\ecal{L}$ is $\otimes$-generating if and only if the restriction of $\ecal{L}$ to the irreducible components of $X$ (with the reduced subscheme structure) are $\otimes$-generating. 
        \item $\ecal{L}$ is $\otimes$-generating if and only if $f^*\ecal{L}$ is $\otimes$-generating. \qedhere
    \end{enumerate}
\end{corollary}

We provide an example to show that (ii) and (iii) do not have analogues for compact generation in general, see Example \ref{example-finitepullbackgenerator}. Namely, there exists a Noetherian scheme $X$ and a perfect complex $G$ on $X$ such that the restriction of $G$ to each irreducible component compactly generates the derived category of the component, but $G$ is not a compact generator. This also gives an example of a finite surjective morphism $f : Y \to X$ where $Lf^*(G)$ compactly generates $D_{\operatorname{QCoh}}(Y)$ but $G$ doesn't generate $D_{\operatorname{QCoh}}(X)$, as we may take $f : Y \to X$ to be the canonical morphism from the disjoint union of the irreducible components.

The results of Corollary \ref{corollary-consequences} are proven in \ref{subsection-someconsequences}, and show that there is a good theory of $\otimes$-generating line bundles analogous to the theory of ample line bundles. We pursue this idea further in Section \ref{section-realdivisors}, where we define and study the cone of $\otimes$-generating $\mathbb{R}$-divisors in the real Neron--Severi space $N^1(X) = NS(X) \otimes \mathbb{R}$ for a proper variety $X$ over an algebraically closed field $k$. In particular, we show $\otimes$-generation of a line bundle is a numerical property in Proposition \ref{prop-numericalstudy}. This establishes explicit relations of $\otimes$-generation with positivity studies in birational geometry and provides a powerful tool for checking whether a line bundle on $X$ is $\otimes$-generating. For example, we provide an intersection theoretic criterion on surfaces as Lemma \ref{lemma-tensoramplenessonsurface} and Lemma \ref{lem: tensor ample cone for surfaces}. We use the criterion to compute the $\otimes$-generating cone of a geometrically ruled surface and discuss $\otimes$-generation on general blowups of $\mathbb{P}^2$, which is related to Nagata's conjecture on plane curves \ref{subsection-surfaces}.

In Section \ref{section-examples}, we use our characterization of $\otimes$-generating line bundles and our study of the $\otimes$-generating cone to give a number of examples of $\otimes$-generating line bundles which are neither ample nor anti-ample. The following is a partial summary.

\begin{theorem} The following schemes admit $\otimes$-generating line bundles which are neither ample nor anti-ample:
    \begin{enumerate}
        \item A reducible union of two copies of the projective line meeting at a node.
        \item Any smooth projective surface containing an integral curve of negative self-intersection.
        \item A blowup of a quasi-projective variety in a closed point.
        \item Affine space of any dimension with doubled origin. In fact, the structure sheaf is $\otimes$-generating but not ample. In particular, a scheme with a $\otimes$-generating line bundle need not be separated and need not satisfy the resolution property.
        \item The blowup of the affine plane at the origin minus a closed point of the exceptional divisor. In fact, the structure sheaf is $\otimes$-generating but not ample in this case. 
        \item Hironaka's smooth proper, non-projective threefold. \qedhere 
    \end{enumerate}
\end{theorem}

Our main motivation for studying the question of Brav was to understand which smooth projective varieties can be reconstructed from their derived categories as in \cite{bondal_orlov_2001}. To this end we prove:  

\begin{theorem} \ 
    \begin{enumerate}  \item Let $X$ be a smooth projective surface over an algebraically closed field such that $\omega_X$ or $\omega_X^{-1}$ is big (e.g. toric surfaces or surfaces of general type). Assume $X$ contains no $(-2)$-curve. Then $\omega_X$ is $\otimes$-generating.
        \item Let $k$ be an algebraically closed field and let $X$ be the blowup of $\mathbb{P}^2_k$ in $n$ closed points. Assume either: (a) all the points lie on a line and $n \neq 3$, or (b) all the points lie on a conic and $n \neq 6$. Then $\omega_X$ is $\otimes$-generating. 
        \item Let $E$ be an elliptic curve over the complex numbers. Let $\ecal{F}$ be a rank two vector bundle on $E$ and let $X = \mathbb{P}_E(\ecal{F})$. Assume $\ecal{F}$ is unstable. Then $\omega_X$ is $\otimes$-generating. 
        \item Let $Y \subset \mathbb{P}^4_{\mathbb{C}}$ be a smooth projective hypersurface of general type over the complex numbers. Assume $Y$ contains a line $\ell \subset \mathbb{P}^4_{\mathbb{C}}$ which moves on $Y$, i.e., $H^0(\ell, \ecal{N}_{\ell/Y}) \neq 0$. If $X$ is the blowup of $Y$ in $\ell$, then $\omega_X$ is $\otimes$-generating.
    \end{enumerate}
   In all of these cases, $X$ has no non-isomorphic Fourier--Mukai partner and every Fourier--Mukai autoequivalence of $D_{\operatorname{perf}}(X)$ is standard.
\end{theorem}

In (i) and (iii), it was already proven in \cite{BPtoric}, \cite{A_nsings}, and \cite{Ellipticruledsurface} that $X$ has no non-isomorphic Fourier--Mukai partner and every Fourier--Mukai autoequivalence of $D_{\operatorname{perf}}(X)$ is standard. In (ii) and (iv) this appears to be new. An argument for studying $\otimes$-generation of the canonical bundle is that it provides a unified approach to all of these results.  In our proof of (iii), we additionally give a complete description of the $\otimes$-generating cone of any geometrically ruled surface over the complex numbers. 

In Section \ref{section-families}, we prove a generalization of Theorem \ref{thm-mainintroversion} to the case of finitely many line bundles on a scheme. This was suggested to us by Johan de Jong when we explained to him the proof of Theorem \ref{thm-mainintroversion}.

\begin{theorem}
\label{thm-familyintroversion}
    Let $X$ be a Noetherian scheme. Let $\ecal{L}_1, \dots , \ecal{L}_n$ be line bundles on $X$. Then the set of line bundles
    $\ecal{L}_1^{e_1} \otimes \cdots \otimes \ecal{L}_n^{\otimes e_n}$ with $e_i \in \mathbb{Z}$ generates $D_{\operatorname{QCoh}}(X)$ if and only if for every integral closed subscheme $Z \subset X$ there exist integers $e_1, \dots , e_n$ and a global section $s \in \Gamma(Z, \ecal{L}_1^{e_1} \otimes \cdots \otimes \ecal{L}_n^{\otimes e_n}|_{Z})$ such that $Z_s$ is non-empty and affine. 
\end{theorem}

Theorems \ref{thm-mainintroversion} and \ref{thm-familyintroversion} and their proofs work verbatim with the word ``scheme'' replaced by ``algebraic space.'' A key technical part of the proofs comes in Section \ref{subsection-big} where we study line bundles $\ecal{L}$ on a Noetherian integral scheme $X$ such that there exists an integer $n > 0$ and a global section $s \in \Gamma(X, \ecal{L}^{\otimes n})$ such that the open $X_s = \{s \neq 0\}$ is non-empty and affine. We prove that such line bundles satisfy many of the same basic properties as big line bundles on a proper variety. Indeed if $X$ is proper over a field then this condition is equivalent to $\ecal{L}$ being big. In particular, the reader who is only interested in Theorem \ref{thm-mainproperversion} and not the more general Theorem \ref{thm-mainintroversion} (but who knows basic facts about big line bundles) can skip much of this. However, Lemma \ref{lemma-sufficestochekirredcomponents} is crucial and we don't know a reference which proves it for a scheme proper over a field.

\medskip\noindent\textbf{Conventions.} For a scheme $X$, we denote $D_{\operatorname{QCoh}}(X)$ the full subcategory of the derived category of the category of $\ecal{O}_X$-modules consisting of objects whose cohomology sheaves are quasi-coherent. When $X$ is Noetherian, this coincides with the derived category of the category of quasi-coherent sheaves \cite[\href{https://stacks.math.columbia.edu/tag/09T4}{Tag 09T4}]{stacks-project}. We denote $D_{\operatorname{perf}}(X) \subset D_{\operatorname{QCoh}}(X)$ the full subcategory of perfect complexes. If $T$ is a closed subset of a scheme $X$ then $D_{\operatorname{QCoh}, T}(X)$ denotes the full subcategory of $D_{\operatorname{QCoh}(X)}$ whose objects restrict to zero in $D_{\operatorname{QCoh}}(X\setminus T)$.

\section{Definitions and preliminaries}

\subsection{\texorpdfstring{$\otimes$-}{Tensor-}generating line bundles}
\label{subsection-tensorample}

Recall that if $\ecal{T}$ is a triangulated category and $S \subset \operatorname{ob}\ecal{T}$, then $\langle S \rangle$ denotes the smallest full triangulated subcategory of $\ecal{T}$ containing $S$ which is thick, i.e., closed under taking direct summands. The set $S$ is said to \emph{classically generate} $\ecal{T}$ if  $\langle S \rangle = \ecal{T}$. On the other hand, the set $S$ is said to \emph{generate} $\ecal{T}$ if for every non-zero object $K$ of $\ecal{T}$, we have $\operatorname{Hom}(G[i], K) \neq 0$ for some $i \in \mathbb{Z}$ and $G \in S$. See \cite{bondal2002generators}. 

\begin{definition}[\cite{ito2025polarizations}]
    Let $X$ be a quasi-compact and quasi-separated scheme. A line bundle $\ecal{L}$ on $X$ is said to be \emph{$\otimes$-generating} if the set $\{\ecal{L}^{\otimes n}\}_{n \in \mathbb{Z}}$ classically generates $D_{\operatorname{perf}}(X)$. That is, 
    \[
    D_{\operatorname{perf}}(X) = \langle \{\ecal{L}^{\otimes n}\}_{n \in \mathbb{Z}} \rangle. \qedhere
    \]
\end{definition}

It is not difficult to see that if $\{\ecal{L}^{\otimes n}\}_{n \in \mathbb{Z}}$ classically generates $D_{\operatorname{perf}}(X)$ then so does a finite subset. 

\begin{lemma}
    Let $X$ be a quasi-compact and quasi-separated scheme. Let $\ecal{L}$ be a line bundle on $X$. If $\ecal{L}$ is $\otimes$-generating, then there exists an integer $N > 0$ such that 
    \[
    D_{\operatorname{perf}}(X) = \langle \ecal{O} \oplus \ecal{L} \oplus \cdots \oplus \ecal{L}^{\otimes N} \rangle. \qedhere
    \]
\end{lemma}

\begin{proof}
    The category $D_{\operatorname{perf}}(X)$ has a classical generator $G$ by \cite{bondal2002generators}. If $\ecal{L}$ is $\otimes$-generating, then $G$ can be built from the objects of the set $\{\ecal{L}^{\otimes n}\}_{n \in \mathbb{Z}}$ in finitely many steps using finite direct sums, shifts, cones, and extraction of direct summands, see \cite[Section 2]{bondal2002generators}. Since only finitely many $\ecal{L}^{\otimes n}$ are used in the construction, we see that 
    $$G \in \langle \{\ecal{L}^{\otimes n}\}_{n \in S} \rangle = \langle \bigoplus _{n \in S} \ecal{L}^{\otimes n} \rangle$$
    for some finite subset $S \subset \mathbb{Z}$. But then since $\langle G \rangle = D_{\operatorname{perf}}(X)$ we see 
    $$
    D_{\operatorname{perf}}(X) = \langle \bigoplus_{n \in S} \ecal{L}^{\otimes n} \rangle. 
    $$
    Using the autoequivalence $- \otimes^{\mathbb{L}} \ecal{L}$ of $D_{\operatorname{perf}}(X)$, we can also arrange that $S \subset \{0 , \dots , N\}$ for some integer $N>0$. 
\end{proof}

The following formulation of $\otimes$-generation is very useful in proofs. 

\begin{lemma}
\label{lemma-equiv}
    Let $X$ be a quasi-compact and quasi-separated scheme. Let $\ecal{L}$ be a line bundle on $X$. The following are equivalent:
    \begin{enumerate}
    \item $\ecal{L}$ is $\otimes$-generating.
    \item For every non-zero object $K \in D_{\operatorname{QCoh}}(X),$ there exist integers $n, i$ and a non-zero morphism $\ecal{L}^{\otimes n}[i] \to K$. \qedhere
    \end{enumerate}
\end{lemma}

\begin{proof}
It is a general fact that a set of compact objects generates a triangulated category if and only if (i) the triangulated category is compactly generated and (ii) the set classically generates the subcategory of compact objects \cite[Lemma 2.2]{Neeman1992}. This applies here as the compact objects of $D_{\operatorname{QCoh}}(X)$ are precisely the perfect complexes and $D_{\operatorname{QCoh}}(X)$ has a compact generator \cite{bondal2002generators}.
\end{proof}

Note that if either $\ecal{L}$ or $\ecal{L}^{-1}$ is ample, then $\ecal{L}$ is $\otimes$-generating. In this generality, this is proven in \cite[\href{https://stacks.math.columbia.edu/tag/0BQQ}{Tag 0BQQ}]{stacks-project}. 

\begin{lemma}
\label{lemma-qaffinepullback}
    If $f:X \to Y$ is a quasi-affine morphism between quasi-compact and quasi-separated schemes and $\ecal L$ is a $\otimes$-generating line bundle on $Y$, then $f^*\ecal L$ is $\otimes$-generating. More generally, if $\{G_i\}_{i \in I}$ is a set of perfect complexes which classically generates $D_{\operatorname{perf}}(Y)$, then the set of perfect complexes $\{Lf^*G_i\}_{i \in I}$ classically generates $D_{\operatorname{perf}}(X)$.
\end{lemma}

In particular, the lemma holds when $f$ is affine or an open immersion.

\begin{proof}
The first statement is a special case of the second so we prove the second. Let $0 \neq K \in D_{\operatorname{QCoh}}(X)$. It suffices by \cite[Lemma 2.2]{Neeman1992} to show $\operatorname{Hom}(Lf^*(G_i)[j], K) \neq 0$ for some $i \in I$ and $j \in \mathbb{Z}$. But $K \neq 0$ implies $Rf_*K \neq 0$ since $f$ is quasi-affine (when $f$ is an open immersion this is because $Lf^* \circ Rf_* \cong \operatorname{id}$ and when $f$ is affine this is by \cite[\href{https://stacks.math.columbia.edu/tag/08I8}{Tag 08I8}]{stacks-project}; in general, $f$ is a composition of such) so there exist $i \in I, j \in \mathbb{Z}$ such that
$$\operatorname{Hom}_{X}(Lf^*G_i[j], K) = \operatorname{Hom}_{Y}(G_i[j], Rf_*K) \neq 0,$$
completing the proof.
\end{proof}

\begin{lemma}
\label{lemma-fieldextensions}
Let $k$ be a field, let $k'/k$ be a field extensions, and let $X$ be a quasi-compact and quasi-separated $k$-scheme. Let $\ecal{L}$ be  a line bundle on $X$ and denote $\ecal{L}_{k'}$ its pullback to $X_{k'}$. Then $\ecal{L}$ is $\otimes$-generating on $X$ if and only if $\ecal{L}_{k'}$ is $\otimes$-generating on $X_{k'}$. 
\end{lemma}

\begin{proof}
    The morphism $X_{k'} \to X$ is affine so it follows from Lemma \ref{lemma-qaffinepullback} that if $\ecal{L}$ is $\otimes$-generating then so is $\ecal{L}_{k'}$. Conversely, if $\ecal{L}_{k'}$ is $\otimes$-generating and $0 \neq K \in D_{\operatorname{QCoh}}(X)$, then there is $n \in \mathbb{Z}$ so that 
    $R\operatorname{Hom}(\ecal{L}^{\otimes n}_{k'}, K_{k'}) \neq 0$, but by flat base change we have
    $$
    R\operatorname{Hom}_{X_k'}(\ecal{L}^{\otimes n}_{k'}, K_{k'}) = R \operatorname{Hom}_X (\ecal{L}^{\otimes n}, K) \otimes ^{\mathbb{L}}_k k'
    $$
    and we conclude since $k'/k$ is faithfully flat.
\end{proof}

\subsection{Affine and quasi-affine complements of divisors}
\label{subsection-big}

In this subsection, we define a notion of big line bundle on any integral, quasi-compact and quasi-separated scheme which generalizes the definition for proper varieties over a field. Deligne's formula in the following forms will be frequently applied here and throughout the paper.

\begin{lemma}[Deligne's Formula]
\label{lemma-deligne}
Let $X$ be a quasi-compact and quasi-separated scheme. Let $\ecal{L}$ be a line bundle on $X$. Let $s \in \Gamma(X, \ecal{L})$ and let $X_s = \{ s\neq 0\} \subset X$ be the locus where $s$ doesn't vanish. Then the following hold:
    \begin{enumerate}
        \item $H^0(X_s, \ecal{O}_{X_s}) = \operatorname{colim}_{n \geq 0} H^0(X, \ecal{L}^{\otimes n})$, where the colimit is over the system
        \begin{equation}
        \label{equn-system}
        \ecal{L} \xrightarrow{s} \ecal{L}^{\otimes 2} \xrightarrow{s} \ecal{L}^{\otimes 3} \xrightarrow{s} \cdots .
        \end{equation}
        \item For any object $K \in D_{\operatorname{QCoh}}(X)$, 
        $$
        \operatorname{Hom}_{D_{\operatorname{QCoh}}(X_s)}(\ecal{O}_{X_s}, K_{X_s}) = \operatorname{colim}_{n \geq 0} \operatorname{Hom}_{D_{\operatorname{QCoh}}(X)}(\ecal{L}^{\otimes -n}, K),
        $$
        where the transition maps of the system are as in (i). \qedhere
    \end{enumerate}
\end{lemma}

\begin{proof}
    Part (ii) follows from the projection formula and the fact that the system (\ref{equn-system}) in $D_{\operatorname{QCoh}}(X)$ has homotopy colimit $Rj_*(\ecal{O}_{X_s})$, where $j : {X_s} \to X$ is the inclusion, see \cite[Lemma 1]{olander2021rouquierdimensionquasiaffineschemes}. 
    Statement (i) follows from (ii) but is also standard, see \cite[\href{https://stacks.math.columbia.edu/tag/01PW}{Tag 01PW}]{stacks-project}. 
\end{proof}

\begin{lemma}
\label{lemma-afffromqaff}
    Let $X$ be a quasi-compact and quasi-separated scheme. Let $\ecal{L}$ be a line bundle on $X$. Let $s \in \Gamma(X, \ecal{L})$. Assume the open $X_s \subset X$ is quasi-affine. Then there are an integer $n > 0$ and finitely many sections $t_1, \dots , t_k \in \Gamma(X, \ecal{L}^{\otimes n})$ such that
    \begin{enumerate}
        \item $X_{t_i}$ is affine for each $i$.
        \item $X_s = \bigcup _{i = 1}^k X_{t_i}$. \qedhere
    \end{enumerate}
\end{lemma}

\begin{proof}
    By definition of quasi-affine, there exist elements $f_1, \dots , f_n \in H^0(X_s, \ecal{O}_{X_s})$ such that $(X_s)_{f_i}$ is affine for each $i$ and $X_s = \bigcup (X_s)_{f_i}$. By Deligne's formula, for $n \gg 0$ there are global sections $t_i \in \Gamma(X, \ecal{L}^{\otimes n})$ such that $t_i = s^n f_i$ holds on $X_s$. Then $X_{st_i} = X_s \cap X_{t_i} = (X_s)_{f_i}$ so the sections $st_i \in \Gamma(X, \ecal{L}^{\otimes n + 1})$ work. 
\end{proof}

\begin{lemma}
\label{lemma-bigequiv}
    Let $X$ be an integral scheme which is quasi-compact and quasi-separated. Let $\ecal{L}$ be a line bundle on $X$. The following are equivalent. 
    \begin{enumerate}
        \item There exists an integer $n > 0$ and a global section $s \in \Gamma(X, \ecal{L}^{\otimes n})$ such that the open $X_s$ is non-empty and quasi-affine.
        \item There exists an integer $n > 0$ and a global section $s \in \Gamma(X, \ecal{L}^{\otimes n})$ such that the open $X_s$ is non-empty and affine.
        \item There exists an integer $n > 0$ and a global section $s \in \Gamma(X, \ecal{L}^{\otimes n})$ such that the open $X_s$ is non-empty and the canonical morphism
        $$
        X_s \to \operatorname{Spec}(H^0(X_s, \ecal{O}_{X_s}))
        $$
        is birational. 
        \item There exists an integer $n > 0$ and a global section $s \in \Gamma(X, \ecal{L}^{\otimes n})$ such that the open $X_s$ is non-empty and the generic fiber of the canonical morphism
        $$
        X_s \to \operatorname{Spec}(H^0(X_s, \ecal{O}_{X_s}))
        $$
        is quasi-affine. 
        \item There exists an integer $n > 0$ and a global section $s \in \Gamma(X, \ecal{L}^{\otimes n})$ such that the open $X_s$ is non-empty and there exists an integral domain $R$ and a morphism
        $$
        X_s \to \operatorname{Spec}(R)
        $$
        whose generic fiber is quasi-affine. 
        \qedhere
    \end{enumerate}
\end{lemma}

To make sense of (iii) and (iv), note that since $X$ is integral, so is the non-empty open $X_s \subset X$. It follows that the ring $H^0(X_s, \ecal{O}_{X_s})$ is an integral domain, so its spectrum indeed has a unique generic point. 

\begin{proof}
    (i) $\implies$ (ii) follows from Lemma \ref{lemma-afffromqaff}. (ii) $\implies$ (iii) is because if $X_s$ is affine, then $X_s \to \operatorname{Spec}(H^0(X_s, \ecal{O}_{X_s}))$ is an isomorphism.
    (iii) $\implies$ (iv) and (iv) $\implies$ (v) are trivial.
    
    Let us show (v) $\implies$ (i). Let $K$ be the fraction field of $R$. Since $K = \operatorname{colim}_{0 \neq f \in R} R[1/f]$
    we have
    $$
    (X_s)_K = \operatorname{lim}_{0 \neq f \in R}(X_s)_f
    $$
    where $(X_s)_f = X_s \otimes _R R[1/f]$ is the locus where the image of $f$ in $H^0(X_s, \ecal{O}_{X_s})$ doesn't vanish. The transition morphisms in the system are affine, and each term is quasi-compact and quasi-separated since $X_s$ is, as the inclusion $X_s \to X$ is affine. By \cite[\href{https://stacks.math.columbia.edu/tag/01Z5}{Tag 01Z5}]{stacks-project}, we see there is a $0 \neq f \in R$ such that $(X_s)_f$ is quasi-affine. By Deligne's formula, there is an integer $m> 0$ and a global section $t \in \Gamma(X, \ecal{L}^{\otimes m \cdot n})$ such that $t = s^mf$ holds on $X_s$ and then $X_{st} = X_s \cap X_t = (X_s)_{f}$ is quasi-affine so the section $st \in \Gamma(X, \ecal{L}^{\otimes (m+1)n})$ works.
\end{proof}

\begin{lemma}\label{lemma: bigequivprop}
    Let $X$ be a 
    variety over a field $k$. Let $\ecal{L}$ be a line bundle on $X$. Then the equivalent conditions of Lemma \ref{lemma-bigequiv} are also equivalent to:
    \begin{enumerate}\setcounter{enumi}{5}
    \item There exists a positive integer $n>0$ such that the natural rational map $X \ratmap \bb PH^0(X, \ecal L^{\otimes n})$ is birational onto its image.
    \end{enumerate}

    If furthermore $X$ is proper over $k$, then all of these conditions are equivalent to: 
    \begin{enumerate}\setcounter{enumi}{6}
        \item $\ecal{L}$ is {big}, i.e., there exist constants $m_0,C>0$ such that $\dim_k \Gamma(X,\ecal L^{\tens m_0m}) > C \cdot m^{\dim X}$ for any $m\gg 0$.
        \qedhere
    \end{enumerate}
\end{lemma}

\begin{proof}
     We first prove (vi) is equivalent to the equivalent conditions of Lemma \ref{lemma-bigequiv}. Assume there exists an $n$ as in the statement of (vi). Then in particular, there is $s \in H^0(X, \mathcal{L}^{\otimes n})$ such that $X_s \neq \emptyset$, that is $s \neq 0$. Then the rational map $X \ratmap \mathbb{P}H^0(X, \ecal{L}^{\otimes n})$ is defined on $X_s$ and factors through a morphism $X_s \to \mathbb{A}^N_k\cong D_+(s) \subset \mathbb{P}H^0(X, \ecal{L}^{\otimes n})$ which is still birational onto its image. If $V \subset \mathbb{A}^N_k$ is the closure of the image, with reduced induced subscheme structure, then $X_s \to V$ is a birational morphism to an integral affine scheme, proving that condition (v) holds.
    Conversely, we show (ii) $\implies$ (vi). Take generators $f_1, \dots, f_m \in H^0(X_s, \ecal O_{X_s})$ (as a $k$-algebra). For $N \gg 0$, we can write
    $$
    f_i = t_i/s^N
    $$
    for some $t_i \in \Gamma(X, \ecal{L}^{\otimes n \cdot N})$ by Deligne's formula. Let $V \subset H^0(X,\ecal L^{\tens n\cdot N})$ be a $k$-linear subspace spanned by $s^N, t_1, \dots, t_m$. Then, note that the natural rational map $X_s \subset X \ratmap \bb P V$ factors through a closed immersion $X_s \inj D_+(s^N) = \bb A (V/k \cdot s^N) \subset \bb P V$ given by the surjection $\operatorname{Sym}(V/k \cdot s^N) \ni \tilde t_i \mapsto f_i \in H^0(X_s, \ecal O_{X_s})$. Hence, we see that the natural rational map
    \[
    X \ratmap \bb PV \ratmap \bb PH^0(X,\ecal L^{\tens n\cdot N}).
    \]
    is birational onto its image.

    From now on assume $X$ is also proper over $k$. We prove (vii) $\implies$ (i). When $X$ is projective, this follows from Kodaira's Lemma. The same argument works in the proper case. Let $W \subset X$ be the complement of a non-empty affine open of $X$ given the reduced subscheme structure and let $\ecal{I}$ be its ideal sheaf. Then for $n \in \mathbb{Z}$ there is an exact sequence of coherent sheaves
    $$
    0 \to \ecal{L}^{\otimes n}\otimes \ecal{I} \to \ecal{L}^{\otimes n} \to \ecal{L}^{\otimes n} \otimes \ecal{O}_W \to 0.
    $$
    There is a constant $D > 0$ such that $H^0(X, \ecal{L}^{\otimes n} \otimes \ecal{O}_W) \leq D \cdot n ^{\operatorname{dim}W}$ for every $n > 0$ so by definition of big line bundle, there is an $n > 0$ and an element  $0 \neq s \in \Gamma(X, \ecal{L}^{\otimes n})$ which restricts to zero on $W$. But this means $ \emptyset \neq X_s \subset X \setminus W$ so $X_s$ is quasi-affine, as needed.
    
    For the other direction we will show (iii) $\implies$ (vii). If (iii) holds, then the fraction field of $H^0(X_s, \ecal{O}_{X_s})$ is equal to $k(X)$ which has transcendence degree $d =\dim X$ over $k$. We can find elements $f_1, \dots , f_d \in H^0(X_s, \ecal{O}_{X_s})$ forming a transcendence basis of its fraction field over $k$. For $N \gg 0$, we can write
    $$
    f_i = t_i/s^N
    $$
    for some $t_i \in \Gamma(X, \ecal{L}^{\otimes n \cdot N})$ by Deligne's formula. Since the $f_i$ are algebraically independent, we see that for every integer $r \geq 0$ the degree $r$ monomials in $s^N, t_1, \dots , t_d$ are linearly independent elements of $\Gamma(X, \ecal{L}^{\otimes r \cdot n \cdot N})$. Hence:
    $$
    \operatorname{dim}_k \Gamma(X, \ecal{L}^{\otimes r \cdot n \cdot N}) \geq \choose{d+r}{r},
    $$
    and the right hand side is a polynomial in $r$ of degree $d$, so $\ecal{L}$ is big.  
\end{proof}

Justified by the above, we make the following definition. 

\begin{definition}
\label{defn-big}
    Let $X$ be an integral scheme which is quasi-compact and quasi-separated. Let $\ecal{L}$ be a line bundle on $X$. We say $\ecal{L}$ is \emph{big} if the equivalent conditions of Lemma \ref{lemma-bigequiv} hold.
\end{definition}

\begin{remark}
\label{remark-bigstructuresheaf}
    Let $X$ be an integral scheme which is quasi-compact and quasi-separated. If $\ecal{O}_X$ is big then any line bundle $\ecal{L}$ on $X$ is big as well: Choose  a global section $f \in H^0(X, \ecal{O}_X)$ such that $X_f$ is non-empty and affine. Then there is $g \in \ecal{O}(X_f)$ such that $(X_f)_g$ is non-empty and affine and $\ecal{L}|_{(X_f)_g}$ is the trivial line bundle. By Deligne's formula, $f^ng$ extends to a global section $h$ of $\ecal{O}_X$, and replacing $f$ by $fh$ we may assume that in addition, $\ecal{L}|_{X_f} \cong \ecal{O}_{X_f}$. Let $t \in \Gamma(X_f, \ecal{L}_{X_f})$ be a nowhere vanishing section. Then by Deligne's formula, for some integer $r > 0$, the section $f^rt$ extends to a global section $s \in \Gamma(X, \ecal{L})$, and then $fs \in H^0(X, \ecal{L})$ is such that $X_{fs} = X_f$ is affine. 
\end{remark}

\begin{lemma}
\label{lemma-pullbackisbig}
    Let $f : Y \to X$ be a morphism of integral, quasi-compact and quasi-separated schemes. Assume the generic fiber of $f$ is quasi-affine. If $\ecal{L}$ is a big line bundle on $X$ then $f^*\ecal{L}$ is a big line bundle on $Y$. 
\end{lemma}

\begin{proof}
    If $\ecal{L}$ is big then there is an integer $n > 0$ and a global section $s \in \Gamma(X, \ecal{L}^{\otimes n})$ such that $X_s$ is non-empty and affine. Then consider the section $f^*s \in \Gamma(Y, f^*\ecal{L}^{\otimes n}).$ There is a morphism 
    $$
    Y_{f^*s} = f^{-1}(X_s)  \to X_{s}
    $$
    from $Y_{f^*s}$ to an affine integral scheme such that the generic fiber is quasi-affine, hence by (v) of Lemma \ref{lemma-bigequiv}, we see that $f^*\ecal{L}$ is big. 
\end{proof}

\begin{lemma}
\label{lemma-blowup}
    Assume $X$ is a Noetherian integral scheme and $f : Y \to X$ is a proper, birational morphism. Let $\ecal{L}$ be a line bundle on $X$. If $f^*\ecal{L}$ is big then so is $\ecal{L}$.
\end{lemma}

\begin{proof}
    By \cite[\href{https://stacks.math.columbia.edu/tag/081T}{Tag 081T}]{stacks-project}, there is a dense open $U \subset X$ and a closed subscheme $Z \subset X$ disjoint from $U$ such that the blowup of $X$ in $Z$ factors through $f : Y \to X$. By Lemma \ref{lemma-pullbackisbig}, if $f^*\ecal{L}$ is big then the pullback of $\ecal{L}$ to the blowup is big. Thus we may assume $f : Y \to X$ is the blowup of $X$ in $Z$. Denote $E \subsetneq Y$ the exceptional divisor.

     There is an integer $n > 0$ and a section $s \in \Gamma(Y, f^*\ecal{L}^{\otimes n})$ such that $Y_s$ is non-empty and affine. Then there is a function $g \in \Gamma(Y_s, \ecal{O}_{Y_s})$ such that $g$ vanishes at every point of the closed set
    $E \cap Y_s \subsetneq Y_s$ and $(Y_s)_g \neq \emptyset$. Arguing as in the proof of Lemma \ref{lemma-afffromqaff}, there is $N>n$ and a global section $t \in \Gamma(X, f^*\ecal{L}^{\otimes N})$ such that $Y_t = (Y_s)_g$. Then replacing $n$ with $N$ and $s$ with $t$ we see we may assume $Y_s \neq \emptyset$ is affine and $Y_s \cap E = \emptyset$. 

    Claim: If $s$ is as above, then for $r \gg 0$ the section $s^r \in \Gamma(Y, f^*\ecal{L}^{\otimes n \cdot r})$ is the pullback of a unique global section of $t \in \Gamma(X, \ecal{L}^{\otimes n \cdot r})$.

    The claim implies the result because then $Y_s = Y_{f^*t} = f^{-1}(X_t) \to X_t$ is an isomorphism since $Y_s \cap E = \emptyset$ and so $X_t$ is affine. 

    Let us prove the claim. Uniqueness of $t$ is because any two such $t$ would agree on the non-empty open $U \subset X$ and $X$ is integral. Thus to prove existence we may work locally and assume $X$ is affine and $\ecal{L} = \ecal{O}_X$. Since $s$ vanishes along $E$ set theoretically, we may after replacing $s$ by a positive power assume $s \in \Gamma(Y, \ecal{O}_Y(-E))$. Then we conclude by Lemma \ref{lemma-proj} below. 
\end{proof}

\begin{lemma}
\label{lemma-proj}
    Let $A$ be a Noetherian ring. Let $J \subset A$ be an ideal. Let $Y \to X$ be the blowup of $X = \operatorname{Spec}(A)$ in the closed subscheme $V(J) \subset X$. Let $E \subset Y$ be the exceptional divisor. Let $f \in \Gamma(Y, \ecal{O}_Y(-E))$. Then there is an integer $r > 0$ such that $f^r \in \Gamma(Y, \ecal{O}_Y(-rE))$ is in the image of the canonical map
    $$
    J^r \to H^0(Y, \ecal{O}_Y(-rE)).
    $$
\end{lemma}

\begin{proof}
    We have $Y = \operatorname{Proj}(\oplus _{n \geq 0} J^n)$ and the canonical maps 
    $$
    J^r \to H^0(Y, \ecal{O}_Y(-r E)) = H^0(\operatorname{Proj}(\oplus _{n \geq 0} J^n), \ecal{O}_{\operatorname{Proj}(\oplus _{n \geq 0} J^n)}(r)) 
    $$
    are isomorphisms for $r \gg 0$ by Serre's Theorems on the cohomology of Proj \cite[\href{https://stacks.math.columbia.edu/tag/0AG7}{Tag 0AG7}]{stacks-project}.
\end{proof}

\begin{lemma}
\label{lemma-finlocfreebig}
    Let $f : Y \to X$ be a finite locally free surjective morphism of integral, quasi-compact and quasi-separated schemes. Let $\ecal{L}$ be a line bundle on $X$. If $f^*\ecal{L}$ is big, then so is $\ecal{L}$.
\end{lemma}

\begin{proof}
    Let $f$ be locally free of rank $r > 0$. There are an integer $n > 0$ and a global section $s \in \Gamma(Y, f^*\ecal{L}^{\otimes n})$ such that $Y_s$ is non-empty and affine. Then the norm $t$ of $s$ (see \cite[\href{https://stacks.math.columbia.edu/tag/0BCX}{Tag 0BCX}]{stacks-project}) is a global section of $\ecal{L}^{\otimes n \cdot r}$ on $X$ and we have
    $$
    X_s \supset X_{f^*t} = f^{-1}(X_t).
    $$
    Then $X_{f^*t}$ is affine since $X_{f^*t} = X_{f^*t} \cap X_s \to  X_s$ is an affine morphism as it is locally the inclusion of a principal open. Since $f$ is finite and surjective, it follows that $X_t$ is affine \cite[\href{https://stacks.math.columbia.edu/tag/01ZT}{Tag 01ZT}]{stacks-project}. Finally, $t \neq 0$ as is verified by checking at the generic point of $X$, hence $X_t \neq \emptyset$ and the proof is complete.
\end{proof}

\begin{prop}
\label{prop-pullbackbigthenbig}
    Let $f : Y \to X$ be a proper, surjective morphism of Noetherian integral schemes. Assume the generic fiber of $f$ is finite. Let $\ecal{L}$ be a line bundle on $X$. If $f^*\ecal{L}$ is big, then so is $\ecal{L}$. 
\end{prop}

\begin{proof}
    There is a non-empty quasi-compact open $U \subset X$ over which $f$ is finite locally free. Then by \cite[\href{https://stacks.math.columbia.edu/tag/0B49}{Tag 0B49}]{stacks-project}, there is a closed subscheme $Z \subset X$ such that, if $X'$ denotes the blowup of $X$ in $Z$ and $Y' \to X'$ the strict transform of $Y \to X$, then $Y' \to X'$ is finite locally free. If $f^*\ecal{L}$ is big then so is its pullback to the (integral) scheme $Y'$ by Lemma \ref{lemma-pullbackisbig}. Now the morphism $Y' \to X$ factors as a finite locally free surjective morphism followed by a blowup, so we conclude by Lemmas \ref{lemma-blowup} and \ref{lemma-finlocfreebig}. 
\end{proof}

\begin{lemma}
\label{lemma-sufficestochekirredcomponents}
    Let $X$ be a Noetherian scheme. Let $\ecal{L}$ be a line bundle on $X$. Let $Z \subset X$ be an irreducible component given the reduced subscheme structure. Assume $\ecal{L}|_{Z}$ is big. Then there is an integer $n>0$ and a section $s \in \Gamma(X, \ecal{L}^{\otimes n})$ such that $X_s$ is non-empty and affine. 
\end{lemma}

\begin{proof}
    First, we may assume $X$ is reduced. This is because if there is a global section of $\ecal{L}^{\otimes m}|_{X_{red}}$ whose non-vanishing locus is affine and non-empty, then some power of it lifts to a global section of a power of $\ecal{L}$ (see proof of \cite[\href{https://stacks.math.columbia.edu/tag/09MS}{Tag 09MS}]{stacks-project}).

    Now let $W \subset X$ be the union of the irreducible components which are not $Z$, given the reduced subscheme structure. 

    Claim 1: There is an integer $k>0$ and a section $\bar{u} \in \Gamma(Z, \ecal{L}^{\otimes k}|_{Z})$ such that $\bar{u}|_{Z \cap W} = 0$ (by $Z \cap W$ we mean the scheme theoretic intersection $Z \times _X W$) and $Z_{\bar{u}}$ is non-empty and affine.

    Proof: Let $s$ be as in the statement of the lemma. Then $Z_s$ is affine and not contained in $W$ so we can find a function $f \in \Gamma(Z_s, \ecal{O}_{Z_s})$ vanishing at every point of $Z_s \cap W$ and such that $D(f) \subset Z_s$ is non-empty. Then by Deligne's formula there is a global section $\bar{u}$ of a power of $\ecal{L}|_{Z}$ such that $Z_{\bar{u}} = D(f)$ (namely $s^Nf$ extends to a global section $v$ of a power of $\ecal{L}|_{Z}$ and then take $\bar{u} = sv$). Then $\bar{u}$ vanishes at every point of $Z \cap W$ so replacing $\bar{u}$ by a positive power we get $\bar{u} |_{Z \cap W} = 0$, as needed.

    Claim 2: There is a unique section $u \in \Gamma(X,\ecal{L}^{\otimes k})$ such that $u|_{Z} = \bar{u}$ and $u|_{W} = 0$. 

    Proof: Uniqueness is because $X$ is reduced and $u$ is determined 
    at every point of $X$. Therefore, we can prove existence locally.
    Thus we may assume $X = \operatorname{Spec}(A)$ is affine and 
    $\ecal{L} = \ecal{O}_X$. Write $I$ (resp. $J$) for the 
    ideal of $Z$ (resp. $W$). Then we have an element $\bar{u} \in 
    A/I$ which is contained in the ideal $(I + J)/I$ and we have to 
    show there exists a lift $u \in A$ of $\bar{u}$ such that $u \in J$, but this is obvious.

    Now taking $t = u$ proves the lemma.
\end{proof}
\begin{lemma}\label{lem: big is open}
    Let $X$ be a proper variety and suppose $\ecal L$ is a big line bundle. For any line bundle $\ecal M$, there exists $m > 0$ such that $\ecal L^{\tens m} \tens \ecal M$ is big.
\end{lemma}
\begin{proof}
    Consider a proper birational surjective morphism $\pi: X' \to X$ with $X'$ a projective variety given by Chow's lemma. By Lemma \ref{lemma-pullbackisbig} and Lemma \ref{prop-pullbackbigthenbig},  $\ecal L^{\tens m} \tens \ecal M$ is big if and only if $\pi^*(\ecal L^{\tens m} \tens \ecal M)$ is big, so we may assume $X$ is projective. Then, for a large enough integer $m>0$, we can write $\ecal L^{\tens m}=\ecal O_X(A)\otimes \ecal O_X(E)$, where $A$ is an ample Cartier divisor on $X$ and $E$ is an effective Cartier divisor on $X$ by \cite{lazarsfeld2017positivity}*{Corollary 2.2.7}. Now, since $\ecal O_X(nA) \tens \ecal M$ is ample for a large enough integer $n > 0$, we see $\ecal L^{\tens nm} \tens \ecal M$ is big again by \cite{lazarsfeld2017positivity}*{Corollary 2.2.7}.  
\end{proof}
\section{Main result}
\label{section-mainresult}

The main result is as follows. 

\begin{theorem}
\label{ref-theoremmain}
    Let $X$ be a Noetherian scheme. Let $\ecal{L}$ be a line bundle on $X$. Then $\ecal{L}$ is $\otimes$-generating if and only if for every integral closed subscheme $Z \subset X$, either $\ecal{L}|_{Z}$ is big or $\ecal{L}^{-1}|_{Z}$ is big. 
\end{theorem}

For the definition of big in this generality, see Definition \ref{defn-big}. 

\begin{example}
\label{example-reducibleconic}
If $k$ is a field, $X$ is a union of two copies of $\mathbb{P}^1_k$ glued along a node, and $\ecal{L}$ is obtained by gluing $\ecal{O}(1)$ on one copy with $\ecal{O}(-1)$ on the other, then $\ecal{L}$ is $\otimes$-generating. 
\end{example}

Here are several equivalent formulations of the criterion.

\begin{lemma} \label{lemma-equivcondit}
    Let $X$ be a Noetherian scheme. Let $\ecal{L}$ be a line bundle on $X$. The following are equivalent:
    \begin{enumerate}
        \item For every integral closed subscheme $Z \subset X$, either $\ecal{L}|_{Z}$ is big or $\ecal{L}^{-1}|_{Z}$ is big. 
        \item For every non-empty closed subscheme $Y \subset X$, there exists an integer $n$ and $s \in \Gamma(Y, \ecal{L}^{\otimes n}|_{Y})$ such that the open $Y_s$ is non-empty and affine.
        \item There exists a stratification
        $$
        X = X^0 \supset X^1 \supset \cdots \supset X^{n+1} = \emptyset
        $$
        of $X$ by closed subschemes $X^i$ such that each $X^i \setminus X^{i+1}$ is affine and equal to the non-vanishing locus of a global section of a power of $\ecal{L}|_{X^i}$. \qedhere
    \end{enumerate}
\end{lemma}

\begin{remark}
    A line bundle $\ecal{L}$ on a scheme $X$ is ample if there is a finite open covering by affine schemes of the form $X_s$ where $s \in \Gamma(X, \ecal{L}^{\otimes n})$ for some $n > 0$. Condition (iii) above is strongly analagous. It replaces ``open covering'' with ``stratification'' and ``$n>0$'' with ``$n \in \mathbb{Z}$.''
\end{remark}

\begin{proof}
(i) $\implies$ (ii) follows from Lemma \ref{lemma-sufficestochekirredcomponents} applied to an irreducible component of $Y$.

(ii) $\implies$ (iii) by induction and the fact that $X$ is Noetherian.

For (iii) $\implies$ (i), let $Z \subset X$ be an integral closed subscheme. Let $i$ be such that $Z \subset X^i$ but $Z \not \subset X^{i+1}$. There is an integer $r$ and $s \in \Gamma(X^i, \ecal{L}^{\otimes r}|_{X_i})$ such that $X^i \setminus X^{i+1} = X^i_s$. Then consider the section $\bar{s} = s|_{Z} \in \Gamma(Z, \ecal{L}^{\otimes k}|_{Z})$.  We have 
$$Z_{\bar{s}} = Z \setminus X^{i+1} \subset X^i \setminus X^{i+1}$$
is quasi-affine, being a closed subscheme of a quasi-affine scheme. But then $\ecal{L}|_{Z}$ or its inverse is big according to whether $r > 0$ or $r <  0$ (if $r  = 0$ then $\ecal{O}_Z$ is big which implies any line bundle on $Z$ is big by Remark \ref{remark-bigstructuresheaf}).
\end{proof}

\subsection{Proof of the ``if'' direction}

\begin{lemma}
\label{lemma-supportedcat}
    Let $X$ be a Noetherian scheme. Let $\ecal{L}$ be a line bundle on $X$. Let $s \in \Gamma(X, \ecal{L})$ be a section and for any integer $n \geq 1$ denote by $D_n$ the zero scheme of $s^n$ (write $D = D_1$). Assume $\ecal{L}|_{D_n}$ is $\otimes$-generating for every $n \geq 1$. Then for any non-zero object $K \in D_{\operatorname{QCoh}, D}(X)$, there exist integers $m, i$ such that $\operatorname{Ext}^i(\ecal{L}^{\otimes m}, K) \neq 0.$
\end{lemma}

\begin{proof}
    Let $K \in D_{\operatorname{QCoh}, D}(X)$. For an integer $r \geq 1$ let $P_r$ be the perfect complex $\ecal{L}^{-r} \xrightarrow{s^r} \ecal{O}_X$. Then there exists an integer $r \geq 1$ such that 
    $$
    P_r^\vee \otimes^{\mathbb{L}} K = R \ecal{H}om_{X}(P_r, K) \neq 0:
    $$
    This is a local question so it reduces to $X$ affine and $\ecal{L} = \ecal{O}_X$. Then $H^i(K) \neq 0$ for some $i$ so there exists a morphism $\varphi: \ecal{O}_X \to K[i]$ corresponding to an element $x \in H^i(K)$. Since $K$ is supported on $D$ we have $s^r x = 0$ for some $r$ which shows that $\varphi$ factors through a (necessarily non-zero) morphism $P_r \to K[i]$. 

    Note that $P_r^\vee \in D^b_{\operatorname{Coh}, D}(X)$ so by \cite[Lemma 7.40]{rouquier2008dimensions}, there exists an integer $n \geq 1$ such that, if $i : D_n \to X$ denotes the inclusion, $P_r^\vee = Ri_*(L)$ for some $L \in D^b_{\operatorname{Coh}}(D_n)$. Then 
    $$
    P_r^\vee \otimes^{\mathbb{L}} K = Ri_*(L) \otimes^{\mathbb{L}} K = Ri_*(L \otimes^{\mathbb{L}}Li^*K).
    $$
    Hence there exists $m$ such that
    $$
    R \operatorname{Hom}(\ecal{L}^{\otimes m}, P_r^\vee \otimes^{\mathbb{L}} K) = R \operatorname{Hom}(\ecal{L}^{\otimes m}, Ri_*(L \otimes^{\mathbb{L}}Li^*K)) = R \operatorname{Hom}(\ecal{L}^{\otimes m}|_{D_n}, L \otimes^{\mathbb{L}}Li^*K) \neq 0,
    $$
    by the assumption that $\ecal{L}|_{D_n}$ is $\otimes$-generating. But
    $$
    R \operatorname{Hom}(\ecal{L}^{\otimes m}, P_r^\vee \otimes^{\mathbb{L}} K) = R \operatorname{Hom}(\ecal{L}^{\otimes m} \otimes^{\mathbb{L}}P_r, K) = R\operatorname{Hom}((\ecal{L}^{m-r} \xrightarrow{s^r}\ecal{L}^m), K)
    $$
    so this implies implies either $R \operatorname{Hom}(\ecal{L}^{\otimes m}, K) \neq 0$ or $R \operatorname{Hom}(\ecal{L}^{\otimes m - r}, K) \neq 0$, as needed.
\end{proof}

\begin{lemma}
\label{lemma-main}
    Let $X$ be a Noetherian scheme. Let $\ecal{L}$ be a line bundle on $X$. Assume there exists an integer $n$ and a section $s \in \Gamma (X, \ecal{L}^{\otimes n})$ such that:
    \begin{enumerate}
        \item The open $X_s = \{s \neq 0\} \subset X$ is quasi-affine, and
        \item For every integer $r \geq 1$, denoting $D_r$ the closed subscheme defined by $s^r$ (and $D = D_1$), the line bundle $\ecal{L}|_{D_r}$ on $D_r$ is $\otimes$-generating. 
    \end{enumerate}
    Then $\ecal{L}$ is $\otimes$-generating.
\end{lemma}

\begin{proof}
Let $K \in D_{\operatorname{QCoh}}(X)$ and assume 
$$
\operatorname{Hom}(\ecal{L}^{\otimes m}[i], K) = 0
$$
for all integers $m, i$. We must show $K = 0$.

First, we claim $K|_{X_s} = 0$. If not, then since $\ecal{L}|_{X_s} = \ecal{O}_{X_s}$ is ample, there exists a nonzero morphism $\varphi: \ecal{O}_{X_s} \to K_{X_s}[i]$ for some $i$. Then by Deligne's formula, there exists an integer $k$ and a morphism
$$
\ecal{L}^{\otimes -nk} \to K[i]
$$
whose restriction to $X_s$ is equal to $\varphi \neq 0$, a contradiction.

Therefore, $K \in D_{\operatorname{QCoh}, D}(X)$. But then Lemma \ref{lemma-supportedcat} shows that $\operatorname{Hom}(\ecal{L}^{\otimes m}[i], K) = 0$ for all $m , i$ implies $K = 0$, as needed.
\end{proof}

\begin{remark}
    We see from the proof that condition (i) can be replaced with the weaker: $\ecal{O}_{X_s}$ is $\otimes$-generating on $X_s$. That this is indeed weaker is discussed in \ref{subsection-tensampstructuresheaf} below. 
\end{remark}

\begin{proof}[Proof of "if" direction of Theorem \ref{ref-theoremmain}]
    Assume the equivalent conditions of Lemma \ref{lemma-equivcondit}. We show by Noetherian induction that for every closed subscheme $Y \subset X$, $\ecal{L}|_{Y}$ is $\otimes$-generating on $Y$. Suppose $Y \subset X$ is a minimal counterexample. Then by condition (ii) there exists a global section $s$ of a power of $\ecal{L}|_{Y}$ whose non-vanishing locus is non-empty and affine. Furthermore, if $D_n \subset Y$ denotes the zero scheme of $s^n$, then $\ecal{L}|_{D_n}$ is $\otimes$-generating for all $n \geq 1$ by our inductive hypothesis. Therefore, by Lemma \ref{lemma-main}, $\ecal{L}|_{Y}$ is $\otimes$-generating, a contradiction.
\end{proof}

\subsection{Proof of the ``only if'' direction}

To prove the "only if" direction of Theorem \ref{ref-theoremmain}, we need to show that a non-zero power $\otimes$-generating line bundle has many non-zero global sections. It is not immediately clear that there are \emph{any}. We produce such sections in Proposition \ref{prop-existsasection}, and the following technical lemma is needed for the proof. 

\begin{lemma}
\label{lemma-pofl}
Let $\ecal{A}$ be an Abelian category. Let $S \subset \operatorname{ob} \ecal{A}$ be a set of non-zero objects of $\ecal{A}$. Assume every morphism between objects in $S$ is either $0$ or an isomorphism. Denote $\ecal{B} \subset \ecal{A}$ the strictly full subcategory whose objects admit finite filtrations
\begin{equation}
\label{equn-filtration}
0 = E_0 \subset E_1 \subset \cdots \subset E_r = E
\end{equation}
such that for each $i = 0 , \dots , r-1$, the quotient $E_{i+1}/E_i$ is isomorphic to an object in $S$. Then $\ecal{B} \subset \ecal{A}$ is a weak Serre subcategory. That is, it is closed under extensions and if $E,F \in \ecal{B}$ and $\varphi : E \to F$ is a morphism in $\ecal{A}$, then $\operatorname{Ker}(\varphi) \in \ecal{B}$ and $\operatorname{Coker}(\varphi) \in \ecal{B}$. 
\end{lemma}

\begin{remark}
\label{remark-rislength}
A consequence of the Lemma is that the integer $r$ in (\ref{equn-filtration}) is well-defined and equal to the length of $E$ in the Abelian category $\ecal{B}$. We cannot use this in the proof however. 
\end{remark}

\begin{proof}
    Closure under extensions is immediate. Assume $E, F \in \ecal{B}$ and $\varphi : E \to F$ is a morphism in $\ecal{A}$.  There are filtrations
    $$
    0 = E_0 \subset E_1 \subset \cdots \subset E_r = E, \hspace{5 em} 0 = F_0 \subset F_1 \subset \cdots \subset F_s = F
    $$
    where each successive quotient is isomorphic to an object in $S$. We prove the $\operatorname{Ker}(\varphi), \operatorname{Coker}(\varphi) \in \ecal{B}$ by double induction on $(r, s)$. When $(r, s) = (1,1)$, it is immediate from the assumption on $S$.

    Next consider the case $r = 1, s$ arbitrary. Consider the composition $\bar{\varphi} : E \to F \to F/F_{s-1}$. This is (isomorphic to) a map between objects in $S$ so is either zero or an isomorphism.

    Case 1: $\bar{\varphi}$ is an isomorphism.

    This means $\varphi: E \to F$ is the inclusion of a direct summand isomorphic to $F/F_{s-1}$. We have $\operatorname{Ker}(\varphi) = 0$ and $\operatorname{Coker}(\varphi) \cong F_{s-1} \in \ecal{B}$. 

    Case 2: $\bar{\varphi} = 0$. 

    Then $\varphi$ factors through $F_{s-1}$ and we have $\operatorname{Ker}(\varphi) = \operatorname{Ker}(E \to F_{s-1})$ and $\operatorname{Coker}(\varphi)$ fits in a short exact sequence
    $$
    0\to \operatorname{Coker}(E \to F_{s-1}) \to \operatorname{Coker}(\varphi) \to F/F_{s-1} \to 0
    $$
    and we conclude that both $\operatorname{Ker}(\varphi)$ and $\operatorname{Coker}(\varphi)$ are in $\ecal{B}$ by induction on $s$.

    Now we do the case where $(r, s)$ are arbitrary. We do this by induction on $r$. Consider the composition $\psi : {E}_1 \to {E} \to {F}$. 

    Case 1: $\psi = 0$. 

    In this case, $\varphi$ factors through a morphism ${E}/{E}_1 \to {F}$ and we have $\operatorname{Coker}(\varphi) = \operatorname{Coker}({E}/{E}_1 \to {F})$ and the kernels fit into a short exact sequence
    $$
    0 \to {E}_1 \to \operatorname{Ker}(\varphi) \to \operatorname{Ker}({E}/{E}_1 \to {F}) \to 0.
    $$
    We conclude by induction on $r$ that both $\operatorname{Ker}(\varphi)$ and $\operatorname{Coker}(\varphi)$ are in $\ecal{B}$. 

    Case 2: $\psi \neq 0$. 

    By the $r = 1$ case we see that $\psi$ is injective. Namely,
    $\operatorname{Ker}(\psi) \subset E_1$ is an object of $\ecal{B}$. Therefore if $\operatorname{Ker}(\psi) \neq 0$ then there are monomorphisms $A \hookrightarrow \operatorname{Ker}(\psi) \hookrightarrow E_1$ where $A \in S$. Since $A \to E_1$ is a monomorphism between objects of $S$ it is an isomorphism. But then $\operatorname{Ker}(\psi) = E_1$ and $\psi = 0$, a contradiction. 
    
    We also have $\operatorname{Coker}(\psi) \in \ecal{B}$ by the $r = 1$ case. Thus, possibly changing $s$ and the filtration ${F}_\bullet$ of ${F}$, we may assume $\varphi$ maps ${E}_1$ isomorphically to ${F}_1$ (explicitly, replace ${F}_1$ with the image of $\psi$ and then get the rest of the filtration by pulling back the filtration on $\operatorname{Coker}(\psi)$). Then applying the snake lemma to the diagram
    \begin{center}\DisableQuotes
    \begin{tikzcd}
        0 \ar[r] & {E}_1 \ar[d,"\cong"] \ar[r] 
        & {E} \ar[d, "\varphi"] \ar[r] & 
        {E}/{E}_1 \ar[d] \ar[r] & 0 \\
        0 \ar[r] &{F}_1 \ar[r] &{F} \ar[r] & {F}/{F}_1 \ar[r] & 0,
    \end{tikzcd}
    \end{center}
    we see that ${E} \to {F}$ has the same
    kernel and cokernel as ${E}/{E}_1 \to 
    {F}/{F}_1$, hence we conclude by 
    induction on $r$. 
\end{proof}

\begin{props}
\label{prop-existsasection}
    Let $X$ be a Noetherian integral scheme over a field $k$ and assume $H^0(X, \ecal{O}_X) = k$. Let $S \subset \operatorname{Pic}(X)$ be a subset. Assume that for every $\ecal{L}, \ecal{M} \in S$ such that $\ecal{L} \not \cong \ecal{M}$, we have $\operatorname{Hom}_{\ecal{O}_X}(\ecal{L}, \ecal{M}) = 0$. If $\langle S \rangle = D_{\operatorname{perf}}(X)$, then $X = \operatorname{Spec}(k)$.
\end{props}

\begin{remark}
\label{remark-torsionlinebundle}
 We will apply this Lemma when $S \subset \operatorname{Pic}(X)$ is a subgroup. In this case, the condition $\operatorname{Hom}(\ecal{L}, \ecal{M}) = 0$ for $\ecal{L} \not\cong \ecal{M}$ can be replaced by $H^0(X, \ecal{L}) = 0$ for $\ecal{L} \not \cong \ecal{O}_X$. This hypothesis holds for example if $S$ is the cyclic subgroup of $\operatorname{Pic}(X)$ generated by a torsion line bundle $\ecal{L}$ on $X$, i.e. $\ecal{L}^{\otimes r} \cong \ecal{O}_X$ for some $r > 0$: If $0 \neq s \in H^0(X, \ecal{L}^{\otimes n})$ then $0 \neq s^r \in H^0(X, \ecal{O}_X) = k$. Hence $s$ vanishes nowhere and $\ecal{L}^{\otimes n} \cong \ecal{O}_X$.
\end{remark}

\begin{proof}
The assumptions imply that the set $S \subset \operatorname{ob} \operatorname{QCoh(X)}$ satisfies the assumption of Lemma \ref{lemma-pofl}. We see that the full subcategory $\ecal{B} \subset \operatorname{QCoh}(X)$ whose objects have a finite filtration with successive quotients isomorphic to a line bundle in $S$ is a weak Serre subcategory. Therefore, the full subcategory
$$
D^b_{\ecal{B}}(X) \subset D_{\operatorname{perf}}(X)
$$
whose objects $K$ satisfy $H^i(K) \in \ecal{B}$ for all $i \in \mathbb{Z}$ is a thick triangulated subcategory. It is triangulated because if $A \to B \to C \to A[1]$ is a distinguished triangle in $D_{\operatorname{perf}}(X)$ and $A, B \in \ecal{T}$, then by the short exact sequences
$$
0 \to \operatorname{Coker}(H^i(A) \to H^i(B)) \to H^i(C) \to \operatorname{Ker}(H^{i+1}(A) \to H^{i+1}(B)) \to 0,
$$
we see that $H^i(C) \in \ecal{B}$ for every $i$, so $C \in \ecal{T}$. It is thick because a direct summand of an object of $\ecal{B}$ is the kernel of a projector and therefore in $\ecal{B}$. 

Since $D^b_{\ecal{B}}(X) \subset D_{\operatorname{perf}}(X)$ is a thick triangulated subcategory containing the generating set $S$, we have
$$
D_{\operatorname{perf}}(X) = D^b_{\ecal{B}}(X).
$$
But this implies in particular that every cohomology sheaf of a perfect complex on $X$ is a vector bundle. Every coherent sheaf on $X$ appears as $H^i(K)$ for some perfect complex $K$ on $X$ (by approximation by perfect complexes, \cite[Theorem 4.1]{approx}). Thus we see that every coherent sheaf is a vector bundle. This can only happen when $X$ is a finite disjoint union of spectra of fields. Since $H^0(X, \ecal{O}_X) = k$ we see that $X = \operatorname{Spec}(k)$.
\end{proof}

We are grateful to Leo Mayer for asking us whether the following is true.

\begin{example} 
\label{example-leo}
Let $X$ be a smooth projective geometrically connected curve over a field. Let $\ecal{L}$ be a line bundle of positive degree on $X$ such that $\Gamma(X, \ecal{L}) = 0$. Then it follows from Proposition \ref{prop-existsasection} that $\ecal{O}_X$ and $\ecal{L}$ do not generate $D_{perf}(X)$.
\end{example}

\begin{corollary}
\label{corollary-steinisbirational}
Let $X$ be an integral Noetherian scheme. Let $G \subset \operatorname{Pic}(X)$ be a subgroup. Assume either:
\begin{enumerate}
    \item For each $\ecal{L} \in G$ with $\ecal{O}_X \not \cong \ecal{L}$, we have $H^0(X, \ecal{L}) = 0$; or
    \item $G$ is the cyclic subgroup generated by a torsion line bundle.
\end{enumerate}
    If $\langle G \rangle = D_{\operatorname{perf}}(X)$, then the canonical morphism
    \begin{equation}
    \label{equn-steinmorphism3}
    X \to \operatorname{Spec}(H^0(X, \ecal{O}_X))
    \end{equation}
    is birational. In particular, in this case $\ecal{O}_X$ is big.
\end{corollary}

Note that $H^0(X, \ecal{O}_X)$ is a domain so its spectrum has a unique generic point.

\begin{proof}
    Let $K$ be the fraction field of the domain $H^0(X, \ecal{O}_X)$. Let $Y \to \operatorname{Spec}(K)$ be the base change of (\ref{equn-steinmorphism3}) to $K$. Then $H^0(Y, \ecal{O}_Y) = K$ by flat base change. The image of $G$ in $\operatorname{Pic}(Y)$ generates $D_{\operatorname{perf}}(Y)$ by Lemma \ref{lemma-qaffinepullback}. It also satisfies the hypotheses of Proposition \ref{prop-existsasection}. In case (i) this is because if $\ecal{O}_X \not \cong \ecal{L} \in G$ and $\ecal{L}_K$ is its pullback to $Y$, then $H^0(Y, \ecal{L}_K) = 0$ by flat base change. In case (ii) this is by Remark \ref{remark-torsionlinebundle}. We therefore conclude that $Y = \operatorname{Spec}(K)$ by Proposition \ref{prop-existsasection} so (\ref{equn-steinmorphism3}) is birational. We see that $\ecal{O}_X$ is big by (iii) of Lemma \ref{lemma-bigequiv} applied to the global section $1 \in \Gamma(X, \ecal{O}_X)$. 
\end{proof}

\begin{proof}[Proof of ``only if'' direction of Theorem \ref{ref-theoremmain}]
    Since the restriction of a $\otimes$-generating line bundle to a closed subscheme is $\otimes$-generating by Lemma \ref{lemma-qaffinepullback}, it suffices to show that if $X$ is an integral Noetherian scheme and $\ecal{L}$ is a $\otimes$-generating line bundle on $X$, then either $\ecal{L}$ or $\ecal{L}^{-1}$ is big. 
    
    Case 1: $H^0(X, \ecal{L}^{\otimes n}) = 0$ for $n \neq 0$.

    Then applying part (i) of Corollary \ref{corollary-steinisbirational}, we see that $\ecal{O}_X$ is big, but then $\ecal{L}$ and $\ecal{L}^{-1}$ are both big by Remark \ref{remark-bigstructuresheaf} and we are done.

    Case 2: $H^0(X, \ecal{L}^{\otimes n}) \neq 0$ for some $n \neq 0$.

    Then take a global section $0 \neq s \in H^0(X, \ecal{L}^{\otimes n})$ for some $n \neq 0$. The line bundle $\ecal{L}|_{X_s}$ is torsion and by Lemma \ref{lemma-qaffinepullback} it is $\otimes$-generating since $X_s \to X$ is affine. Therefore by part (ii) of Corollary \ref{corollary-steinisbirational}, the morphism
    \begin{equation*}
    \label{equn-steinmorph2} 
    X_s \to \operatorname{Spec}(H^0(X_s, \ecal{O}_{X_s}))
    \end{equation*}
    is birational. By Lemma \ref{lemma-bigequiv}, we see that either $\ecal{L}$ or $\ecal{L}^{-1}$ is big depending on whether $n> 0$ or $n<0$. 
\end{proof}

\subsection{Some consequences}
\label{subsection-someconsequences}

\begin{lemma}\label{lem: tensor power of tensor ample}
    Let $\ecal L$ be a line bundle on a Noetherian scheme. Then the following are equivalent.
    \begin{enumerate}
        \item $\ecal L$ is $\otimes$-generating,
        \item $\ecal L^{\tens n}$ is $\otimes$-generating for some $0 \neq n \in \bb Z$,
        \item $\ecal L^{\tens n}$ is $\otimes$-generating for any $0 \neq n \in \bb Z$. \qedhere
    \end{enumerate}
\end{lemma}

\begin{proof}
    This follows from the fact that a line bundle $\ecal{L}$ on an integral quasi-compact and quasi-separated scheme is big if and only if a positive tensor power is, which is immediate from Definition \ref{defn-big}.
\end{proof}

\begin{lemma}\label{lemma-irreduciblecomp}
    Let $X$ be a Noetherian scheme. Let $\ecal{L}$ be a line bundle on $X$. Then $\ecal{L}$ is $\otimes$-generating if and only if so is the restriction of $\ecal{L}$ to every irreducible component (with the reduced subscheme structure).
\end{lemma}

\begin{proof}
    The criterion of Theorem \ref{ref-theoremmain} only depends on the integral closed subschemes of $X$.
\end{proof}

More generally we have:

\begin{lemma}
\label{lemma-finitesurjective}
    Let $f : Y \to X$ be a finite surjective morphism of Noetherian schemes. Let $\ecal{L}$ be a line bundle on $X$. Then $\ecal{L}$ is $\otimes$-generating if and only if so is $f^*\ecal{L}$. 
\end{lemma}

\begin{remark}
If $X$ is proper over a Noetherian ring then the Lemma also holds with ``$\otimes$-generating'' replaced by ``ample,'' see \cite[\href{https://stacks.math.columbia.edu/tag/0B5V}{Tag 0B5V}]{stacks-project}. However, there is an example of a variety $X$ over a field whose normalization is quasi-affine but such that $X$ itself is not quasi-affine, see \cite[\href{https://stacks.math.columbia.edu/tag/0272}{Tag 0272}]{stacks-project}. Thus the line bundle $\ecal{O}_X$ becomes ample after a finite pullback, but is not itself ample. On the other hand, we see from Lemma \ref{lemma-finitesurjective} that $\ecal{O}_X$ is $\otimes$-generating.
\end{remark}

\begin{proof}
    The ``only if'' direction is Lemma \ref{lemma-qaffinepullback}. For the ``if'' direction, assume $f^*\ecal{L}$ is $\otimes$-generating. Let $Z \subset X$ be an integral closed subscheme. Then there is an irreducible component $W \subset Z \times _X Y$ which maps surjectively to $Z$ (the closure of any point in the fiber over the generic point of $Z$ works). Then $f^*\ecal{L}|_{W}$ or $f^*\ecal{L}^{-1}|_{W}$ is big by Theorem \ref{ref-theoremmain} and $W \to Z$ is finite surjective so by Proposition \ref{prop-pullbackbigthenbig}, we see that $\ecal{L}|_{Z}$ is big or $\ecal{L}^{-1}|_{Z}$ is big and we conclude. 
\end{proof}

\begin{example}
\label{example-finitepullbackgenerator}
    There exists a finite surjective morphism of Noetherian schemes $f : Y \to X$ and a perfect complex $P$ on $X$ such that $Lf^*P$ is a compact generator of $D_{\operatorname{QCoh}}(Y)$ but $P$ is not a compact generator of $D_{\operatorname{QCoh}}(X)$. For an example, let $X$ be the union of two copies of the projective line over an algebraically closed field $k$ which meet at one point at which they are tangent. For concreteness, let $X$ be the compactification of the affine reducible curve $\operatorname{Spec}(k[x,y])/(y(y-x^2))$ obtained by adding a smooth point at infinity to each irreducible component. Let $A, B \cong \mathbb{P}^1_k$ be the two irreducible components and choose closed points $a \in A \setminus B$ and $b \in B \setminus A$. Let $\ecal{L}$ be the line bundle $\ecal{O}_X(a-b)$. Let $G \in D_{\operatorname{perf}}(X)$ be the object $\ecal{O}_X \oplus \ecal{L}$. Then the restrictions $G|_{A}$ and $G|_{B}$ are identified with the objects $\ecal{O}_{\mathbb{P}^1_k}\oplus \ecal{O}_{\mathbb{P}^1_k}(1)$ and $\ecal{O}_{\mathbb{P}^1_k}\oplus \ecal{O}_{\mathbb{P}^1_k}(-1)$ which are known to generate $D_{\operatorname{QCoh}}(\mathbb{P}^1_k)$. However, we will show that $G$ is not a compact generator of $D_{\operatorname{QCoh}}(X)$. This suffices as we may take $f : Y \to X$ to be the canonical finite surjective morphism $A \coprod B \to X$. 

    We claim that $H^0(X, \ecal{L}) = 0$. Assume not and let $0 \neq s \in H^0(X, \ecal{L})$. Then $s|_{B} = 0$ since $\ecal{L}|_{B} = \ecal{O}_{\mathbb{P}^1_k}(-1)$. For this to happen, $s|_{A}$ must vanish on the scheme theoretic intersection $A \times _X B \subset A$, which means it must vanish to order $\geq 2$ at the unique closed point in $A \cap B$. But this implies $s|_{A} = 0$ since $\ecal{L}|_{A} = \ecal{O}_{\mathbb{P}^1_k}(-1)$. But then $s = 0$ as $X$ is reduced and $s$ vanishes at each point of $X$.

    By symmetry, we have $H^0(X, \ecal{L}^{-1}) = 0$, and we also have $H^0(X, \ecal{O}_X) = k$ since $X$ is a proper, connected, and reduced scheme over an algebraically closed field. It follows from Proposition \ref{prop-existsasection} that $\ecal{O}_X$ and $\ecal{L}$ do not generate $D_{\operatorname{QCoh}}(X)$. This completes the proof because $\ecal{O}_X, \ecal{L}$ generate $D_{\operatorname{QCoh}}(X)$ if and only if the one object $\ecal{O}_X \oplus \ecal{L}$ does.
\end{example}

After hearing the example above, Kuznetsov pointed out to us that if $X \subset \mathbb{P}^2_k$ is a reducible conic with irreducible components $A, B \cong \mathbb{P}^1_k$, and $a \in A \setminus B$ and $b \in B \setminus A$ are $k$-points, then $G = \ecal{O}_X(a) \oplus \ecal{O}_X(b)$ also gives an example of a perfect complex whose restriction to the irreducible components is a compact generator but such that $G$ is not. In fact, $\operatorname{Ext}^i_X(G, \ecal{O}_X) = 0$ for all $i$ in this case.

In practice, the following seems to be a useful way to check if a line bundle is $\otimes$-generating. 

\begin{lemma}
\label{lemma-usefulcriterion}
    Let $X$ be a Noetherian scheme. Let $\ecal{L}$ be a line bundle on $X$. Let $n_1, \dots, n_k \in \mathbb{Z}$ be integers and  $s_i \in \Gamma(X, \ecal{L}^{\otimes n_i})$ global sections. Assume:
   \begin{enumerate}
        \item The schemes $X_{s_i}$ have $\otimes$-generating structure sheaf (for example if each $X_{s_i}$ is quasi-affine).
        \item The restriction of $\ecal{L}$ to the reduced closed subscheme $V(s_1, \dots , s_k) _{red} \subset X$ cut out by $s_1, \dots , s_k$ is $\otimes$-generating.  
    \end{enumerate}
   Then $\ecal{L}$ is $\otimes$-generating. 
\end{lemma}

\begin{proof}
    Let $Z \subset X$ be an integral closed subscheme. If $Z \subset V(s_1, \dots , s_k)$ then $\ecal{L}|_{Z}$ or $\ecal{L}^{-1}|_Z$ is big by (ii). If not, then for some $i$ the restriction $s_i|_{Z}$ is not zero, and then $Z_{s_i} \subset X_{s_i}$ is quasi-affine. We conclude by Theorem \ref{ref-theoremmain}. 
\end{proof}

\begin{lemma}
\label{lemma-nefandtensorample}
    Let $X$ be a proper scheme over a field. Let $\ecal{L}$ be a line bundle on $X$. If $\ecal L$ is nef and $\otimes$-generating, then it is ample. 
\end{lemma}
\begin{proof}
    The inverse of a big line bundle on a proper variety of positive dimension is not nef, so Theorem \ref{ref-theoremmain} implies $\ecal{L}|_{Z}$ is big for every closed subvariety $Z \subset X$. This implies $\ecal{L}$ is ample: We prove this using the Nakai--Moishezon criterion. We must show
    \begin{equation}
    \label{equn-nakaimoishezon}
    (\ecal{L}^{\operatorname{dim}Z} \cdot Z) > 0
    \end{equation}
    for all closed subvarieties of $X$. But $\ecal{L}|_Z$ is a big and nef line bundle on the proper variety $Z$ so this is true by \cite[Chapter VI, Theorem 2.15]{Kollarrationalcurves}.
\end{proof}

\begin{example} \ 
\label{example-strictlynefnotample}
\begin{enumerate}
    \item A weak Fano variety has $\otimes$-generating canonical bundle if and only if it is a Fano variety.
    \item  If $\ecal L$ is a strictly nef line bundle that is not ample, then $\ecal L$ is not $\otimes$-generating. In particular, $\otimes$-generation cannot be checked only on curves even over a surface. See \cite[Example 1.5.2]{lazarsfeld2017positivity} for a concrete example due to Mumford.  \qedhere
\end{enumerate}
\end{example}

\section{\texorpdfstring{$\tens$}{Tensor}-generating \texorpdfstring{$\bb R$-}{real }divisors}
\label{section-realdivisors}
Theorem \ref{ref-theoremmain} suggests there is a numerical approach to $\otimes$-generating line bundles. In this section, we work with proper varieties over an \emph{algebraically closed} field. This is not a strong restriction by Lemma \ref{lemma-fieldextensions}. We cite results of \cite{lazarsfeld2017positivity}*{Section 1.3} noting \cite{lazarsfeld2017positivity}*{Remark 1.3.16}.
\begin{notation}
    Let $X$ be a proper variety. 
    \begin{itemize}
        \item In this paper, a \emph{divisor} on $X$ means a Cartier divisor on $X$ and we denote $\operatorname{Div}(X)$ the group of divisors. 
        \item A \emph{$\bb Q$-divisor} (resp. \emph{$\bb R$-divisor}) means an element of $\operatorname{Div}(X) \otimes \mathbb{Q}$ (resp. $\operatorname{Div}(X) \otimes \mathbb{R}$).
        \item A \emph{$\bb Q$-line bundle} (resp. \emph{$\bb R$-line bundle}) means an element of $\pic(X) \tens \bb Q$ (resp. $\pic(X)\tens \bb R$).
        \item  For a $\bb Q$-divisor (resp. an $\bb R$-divisor) $D$, we write its image under the natural surjective map $\operatorname{Div}(X) \tens \bb Q \surj \pic(X)\tens \bb Q$ (resp. $\operatorname{Div}(X) \tens \bb R \surj \pic(X)\tens \bb R$) by $\ecal O_X(D)$.
        \item For a morphism $f:X\to Y$ of proper varieties and  a $\bb Q$-line bundle (resp. an $\bb R$-line bundle) $\ecal L$ on $Y$, we write its image under the natural map $f^*: \pic(Y)\tens \bb Q \to \pic(X)\tens \bb Q$ (resp. $f^*: \pic(Y)\tens \bb R \to \pic(X)\tens \bb R$) by $f^*\ecal L$. 
        \item We will sometimes refer to a divisor as an \emph{integral divisor} to differentiate from $\mathbb{Q}$- or $\mathbb{R}$-divisors. This does not mean integral in the sense of integral scheme. 
        \qedhere
    \end{itemize}
    
\end{notation}
Properties of integral divisors can be naturally generalized to $\bb Q$-divisors. 
\begin{definition}
    Let $X$ be a proper variety. We say a $\bb Q$-line bundle $\ecal L$ is \emph{ample} (resp. \emph{big}, resp. \emph{$\otimes$-generating}, resp. \emph{nef}) if there exists an integer $m > 0$ such that  $\ecal L^{\tens m} \in \pic(X)$ and $\ecal L^{\tens m}$ is {ample} (resp. {big}, resp. {$\otimes$-generating}, resp. nef). We say a $\bb Q$-divisor $D$ is \emph{ample} (resp. \emph{big}, resp. \emph{$\otimes$-generating}, resp. \emph{nef}) if so is $\ecal O_X(D)$.  Similarly, we say a $\bb Q$-divisor $D$ is \emph{effective} if there exists an integer $m \geq 0$ such that $mD \in \operatorname{Div}(X)$ is effective. 
\end{definition}

\begin{remark}\label{remark-divisor with affine complements}
    In the language of divisors, Lemma \ref{lemma-bigequiv} says a divisor $D$ on a proper variety is big if and only if for a large enough integer $n> 0$, $nD$ is linearly equivalent to an effective divisor with affine complement. 
\end{remark}

Let us fix several definitions about $\bb R$-divisors on a proper variety for completeness. 
\begin{construction}
 Let $X$ be a proper variety. Following conventions in \cite{lazarsfeld2017positivity}*{Section 1.3.B}, we say:
 \begin{itemize}
     \item Two $\bb R$-divisors $D,E$ are \emph{linearly equivalent} if $\ecal O_X(D) = \ecal O_X(E)$ in $\pic(X) \tens \bb R$, in which case we write $D\equiv_\sf{lin} E$. For a morphism $f:X \to Y$ of proper varieties and for an $\bb R$-divisor $D$ on $Y$, we say $f^*D$ is \emph{linearly equivalent} to an $\bb R$-divisor $E$ on $X$ if $f^* \ecal O_Y (D) = \ecal O_X(E)$ in $\pic(X) \tens \bb R$, in which case we write $f^* D \equiv_\sf{lin} E$ by slight abuse of notations. When $f:Z \inj X$ is a closed immersion, we may write $D|_Z$ instead of $f^*D$ (again defined up to linear equivalence). 
     \item An $\bb R$-line bundle $\ecal{L}$ is \emph{ample} (resp. \emph{big}) if we can write $\ecal{L} = \sum a_i \ecal{L}_i$ in $\operatorname{Pic}(X) \tens \bb R$ with ample (resp. big) line bundles $\ecal{L}_i \in \operatorname{Pic}(X)$ and real numbers $a_i > 0$.  
     \item An $\bb R$-divisor $D$ is \emph{ample} (resp. \emph{big}) if so is the $\mathbb{R}$-line bundle $\ecal{O}_X(D)$. 
     \item An $\bb R$-divisor $D$ is \emph{effective} if we can write $D = \sum a_i D_i$ in $\operatorname{Div}(X) \tens \bb R$ with effective divisors  $D_i \in \operatorname{Div}(X)$ and real numbers $a_i > 0$.
     \item An $\bb R$-divisor $D$ is \emph{$\otimes$-generating} if for any closed subvariety $Z \subset X$, $D|_Z$ is linearly equivalent to a big or anti-big $\bb R$-divisor $D_Z$ on $Z$, i.e., $\ecal O_X(D)|_Z = \ecal O_Z(D_Z)$ for a big or anti-big $\bb R$-divisor $D_Z$ on $Z$. 
 \end{itemize}
 In particular, ampleness, bigness, and $\otimes$-generation are defined a priori up to linear equivalence, whereas effectiveness is not. 
For an $\bb R$-divisor $D = \sum a_i D_i$ with $D_i \in \operatorname{Div}(X)$ and a closed subvariety $Z \subset X$, define the intersection product $(D^{\dim Z} \cdot Z) := \sum a_i (D_i^{\dim Z} \cdot Z)$. It is easy to see from the classical theory that if $D$ and $D'$ are linearly equivalent $\bb R$-divisors on $X$, then for any closed subvariety $Z \subset X$, $(D^{\dim Z} \cdot Z) = ({D'}^{\dim Z} \cdot Z)$. 
 \begin{itemize}
     \item Two $\bb R$-divisors $D,E$ are \emph{numerically equivalent} if $(D^{\dim Z}\cdot Z) = (E^{\dim Z} \cdot Z)$ for any closed subvariety $Z \subset X$, in which case we write $D \equiv_\sf{num} E$. By Kleiman's theorem (\cite{Kollarrationalcurves}*{Chapter VI, Theorem 2.17}), it is indeed enough to check against integral curves to see numerical equivalence of $\bb R$-divisors on a proper variety. 
     \item An $\bb R$-divisor is \emph{nef} (resp. \emph{numerically trivial}) if $(D \cdot C) \geq 0$ (resp.$(D \cdot C) = 0$) for any irreducible curve $C \subset X$.
 \end{itemize}
Define the \emph{real Neron-Severi space} of $X$ to be $N^1(X)  := \operatorname{Div}(X) \otimes \bb R/\equiv_\sf{num}$. By the Theorem of the Base, $N^1(X)$ is finite dimensional. For an $\bb R$-divisor $D$ on $X$, we write its numerical class by $[D] \in N^1(X)$. For a morphism $f:X\to Y$ of proper varieties, note we have a well-defined homomorphism $f^*:N^1(Y) \to N^1(X)$. 
\end{construction}
The following well-known fact gets rid of possible ambiguity in the notion of numerical equivalence. It can be shown in the same way as \cite{fujino2012fundamental}*{Lemma 7.11} or \cite[Proof of Proposition 1.3.13]{lazarsfeld2017positivity} since the Theorem of the Base works over any algebraically closed field.  
\begin{lemma}\label{lemma-numerical}
    Let $X$ be a proper variety and let $D$ be a numerically trivial $\bb R$-divisor. Then, there are finitely many numerically trivial integral divisors $D_i$ and real numbers $a_i \in \bb R$ such that $D = \sum_i a_i D_i$ in $\operatorname{Div}(X)\tens \bb R$. In particular, we have $N^1(X) = (\operatorname{Div}(X)/\equiv_{\sf{num},\bb Z})\tens \bb R$, where $\equiv_{\sf{num},\bb Z}$ denotes numerical equivalence on integral divisors. 
\end{lemma}
The following is fundamental for big $\bb R$-divisors.
\begin{lemma}[Kodaira's lemma]\label{lem: Kodaira}
    Let $X$ be a proper variety and suppose $D$ is a big $\bb R$-divisor. Then, for any effective Cartier divisor $E \subset X$, there is $m>0$ such that $mD - E$ is linearly equivalent to an effective $\bb R$-divisor. 
\end{lemma}
\begin{proof}
   Write $D = \sum a_i D_i$ with $a_i > 0$ positive real numbers and $D_i$ big integral divisors. Replacing $D$ with $m D$ for $m > 0$ if necessary we may assume each $D_i$ is linearly equivalent to an effective divisor $D_i'$, and then replacing $D_i$ with $D_i'$, we may assume $D = \sum a_i D_i$ with $a_i >0$ real numbers and $D_i$ big and effective integral divisors. Next, choosing rational numbers $0 < a_i' < a_i$ and setting $D' = \sum a_i' D_i$, the divisor $D - D'$ is effective, so it suffices to prove the Lemma for the $\mathbb{Q}$-divisor $D'$ and we may replace $D$ with $D'$. Finally, replacing $D$ with $m D$ for $m > 0$ again we may assume $D$ is a big Cartier divisor. Then, it suffices to show $H^0(X,\ecal O_X(mD -E)) \neq 0$ for $m\gg 0$, which follows by the same arguments as in the projective case (see the proof of Lemma \ref{lemma: bigequivprop}).
\end{proof}

Now, we note that the definitions above are compatible with the ones for $\bb Q$-divisors. 
\begin{lemma}\label{lemma-tensampleasR-divisQ-div}
    Let $X$ be a proper variety and let $D$ be a $\bb Q$-divisor. 
    Then $D$ is ample (resp. big, resp. nef, resp. $\otimes$-generating) as a $\bb Q$-divisor if and only if $D$ is ample (resp. big, resp. nef, resp. $\otimes$-generating) as an $\bb R$-divisor. 
\end{lemma}
\begin{proof}
    The ``only if'' directions are obvious, as is the ``if'' direction when $D$ is assumed nef. Suppose $D$ is ample as an $\bb R$-divisor. Then for any closed subvariety $Z\subset X$, we have $(D^{\dim Z}\cdot Z) > 0$. Thus, by the Nakai--Moishezon criteria on a proper variety (\cite{Kollarrationalcurves}*{Chapter VI, Theorem 2.15}), $D$ is ample as a $\bb Q$-divisor.
    
    Now, suppose $D$ is big as an $\mathbb{R}$-divisor. We can write $D = \sum a_i B_i$ with big divisors $B_i \in \operatorname{Div}(X)$ and real numbers $a_i > 0$. Take a proper birational morphism $\pi: \tilde X \to X$ with $\tilde X$ projective given by Chow's lemma and take $\tilde B_i \in \operatorname{Div}(\tilde X)$ so that $\ecal O_{\tilde X}(\tilde B_i) = \pi^*\ecal O_X(B_i)$ in $\pic(\tilde X)$. In particular, $\pi^*\ecal O_X(D) = \ecal O_{\tilde X}(\sum a_i \tilde B_i)$ in $\pic(\tilde X)\tens \bb R$.
    Now, by Lemma \ref{lemma-pullbackisbig}, $\sum a_i \tilde B_i$ is a big $\bb R$-divisor, so applying Lemma \ref{lem: Kodaira} to an effective ample divisor $H$ on $\tilde{X}$, there is an integer $n > 0$ such that $n \sum a_i \tilde{B}_i - H$ is linearly equivalent to an effective $\mathbb{R}$-divisor $E$. Write $E = \sum b_i E_i$ with $b_i > 0$ and $E_i$ effective divisors. Then approximating $b_i$ by rational numbers and using openness of $\bb R$-amplitude (e.g. \cite{lazarsfeld2017positivity}*{Example 1.3.14}), we may write $n \sum a_i \tilde{B}_i \equiv_\sf{lin} A + E'$ where $A$ is an ample $\mathbb{R}$-divisor and $E' = \sum b_i' E_i$ where $b_i'$ are positive rational numbers close to $b_i$ and $E_i$ are effective divisors. In particular, we see that $E'$ is effective as a $\mathbb{Q}$-divisor. 
    This forces $\ecal O_{\tilde X}(A)$ to be a $\bb Q$-line bundle and by the first paragraph, we see that the $\bb Q$-line bundle $\ecal O_{\tilde X}(A)$ is ample (as a $\bb Q$-line bundle). Thus, $\ecal O_{\tilde X}(n\sum a_i \tilde B_i) = \pi^*\ecal O_X(nD)$ is a big $\bb Q$-line bundle and we are done by Lemma \ref{prop-pullbackbigthenbig} 

    Finally, suppose a $\bb Q$-divisor $D$ is $\otimes$-generating as an $\bb R$-divisor. Choose an integer $m > 0$ such that $m D \in \operatorname{Div}(X)$ and set $\ecal{L} = \ecal{O}_X(m D)$. Then by definition of $\otimes$-generation for $\mathbb{R}$-divisors, $\ecal{L}|_Z$ is either big or anti-big as an $\mathbb{R}$-line bundle, and by the previous paragraph, $\ecal{L}|_{Z}$ is either big or anti-big as a $\mathbb{Q}$-line bundle. Since a line bundle is big if and only if a positive tensor power is, this means that $\ecal{L}|_{Z}$ is either big or anti-big as a line bundle, and so $\ecal{L}$ is $\otimes$-generating.
\end{proof}

\begin{remark}
    The ample case of the Lemma  also follows from the Nakai--Moishezon criterion for $\bb R$-divisors on a proper variety given in \cite{fujino2023nakai}.
\end{remark}

Next, we show some notions of positivity only depend on numerical classes. Namely:
\begin{lemma}\label{lemma-numericalequiv}
    Let $X$ be a proper variety. Suppose $D$ and $D'$ are numerically equivalent $\bb R$-divisors on $X$. If $D$ is ample (resp. big, resp. $\otimes$-generating), then so is $D'$.
\end{lemma}
\begin{proof}
    If $D$ is an ample $\bb R$-divisor, then $X$ is projective and this follows from \cite{lazarsfeld2017positivity}*{Proposition 1.3.13}. 
    
    Next, suppose $D$ is a big $\bb R$-divisor. The claim follows by \cite{fujino2012fundamental}*{Proposition 7.12} if we work over an algebraically closed field of characteristic zero. Let us give arguments that work over any algebraically closed field although the idea is essentially the same. More precisely, instead of taking a resolution of singularities as in \cite{fujino2012fundamental}*{Proposition 7.12} (where the hypothesis on the ground field is used), we will use Chow's lemma. First, by the same arguments as in the proof of \cite{fujino2012fundamental}*{Proposition 7.12.} (using Lemma \ref{lemma-numerical} instead of \cite{fujino2012fundamental}*{Lemma 7.11}), it suffices to show $B + r G$ is big, where $r\in \bb Q$, $B$ is a big integral divisor, and $G$ is a numerically trivial integral divisor (see also \cite[Proof of Proposition 1.3.13]{lazarsfeld2017positivity}). Now, we may assume $X$ is projective by Chow's lemma as in Lemma \ref{lem: big is open} (noting the pull-back of a numerically trivial divisor is numerically trivial), but then we can write $B = A + E$ where $A, E$ are $\mathbb{Q}$-divisors with $A$ ample and $E$ effective and so $B + rG = (A + rG) + E$, and this is big because $A + rG$ is an ample $\mathbb{Q}$-divisor by \cite[Proposition 1.3.13]{lazarsfeld2017positivity}). See \cite{fujino2012fundamental}*{Proposition 7.12}. 

    Finally, suppose $D$ is a $\otimes$-generating $\bb R$-divisor. Then, for any closed subvariety $Z \subset X$, we can take $\bb R$-divisors $D_Z$ and $D'_Z$ on $Z$ such that $\ecal O_Z(D_Z) = \ecal O_X(D)|_Z$ and $\ecal O_Z(D'_{Z}) = \ecal O_X(D')|_Z$. Then by supposition $D_Z \equiv_\sf{num} D_Z'$ and $D_Z$ is big or anti-big. Hence, by the big case, $D_Z'$ is big or anti-big and we are done.   
\end{proof}
\begin{remark}
    In some literature (including \cite{fujino2012fundamental}), a big line bundle on a proper variety is defined to be a line bundle whose pullback to its normalization is big in our sense (Definition \ref{defn-big}). Our definition and their definition are equivalent by Lemma \ref{lemma-pullbackisbig} and Proposition \ref{prop-pullbackbigthenbig} together with \cite[\href{https://stacks.math.columbia.edu/tag/035Q}{Tag 035Q}]{stacks-project} and \cite[\href{https://stacks.math.columbia.edu/tag/035S}{Tag 035S}]{stacks-project}. More generally, if $\nu: X^\nu \to X$ is a normalization of a noetherian integral Nagata scheme $X$, then a line bundle $\ecal L$ on $X$ is big if and only if so is $\nu^*\ecal L$. 
\end{remark}
Now, we have justified the definition of the following cones in the real Neron-Severi space. 
\begin{definition}\label{defn-tensamplecone}
    Let $X$ be a proper variety. Define $\operatorname{Amp}(X)$ (resp. $\operatorname{Big}(X)$, resp. $\operatorname{Nef}(X)$, resp. $\tensgen(X)$) to be the cone of ample (resp. big, resp. nef, resp. $\otimes$-generating) $\bb R$-divisors up to numerical equivalence sitting inside of $N^1(X)$. Also, define the cone $\operatorname{PsEff}(X)$ of pseudo-effective $\bb R$-divisors on $X$ to be the closure of the big cone in $N^1(X)$. For an $\bb R$-divisor $D$ on $X$, we say its class $[D] \in N^1(X)$ is ample (resp. big, resp. nef, resp. $\otimes$-generating) if $[D]$ lies in the respective cone. 
\end{definition}
As a direct consequence of Lemma \ref{lemma-tensampleasR-divisQ-div} and Lemma \ref{lemma-numericalequiv}, we get the following, which justifies the numerical study of $\otimes$-generating line bundles. 
\begin{prop}
\label{prop-numericalstudy}
    Let $X$ be a proper variety and let $D$ be an integral divisor with class $[D] \in N^1(X)$. Then $\ecal{O}_X(D)$ is ample (resp. big, resp. $\otimes$-generating) if and only if $[D]$ is ample (resp. big, resp. $\otimes$-generating).
\end{prop}
Also, the following generalization of a classical result holds.
\begin{lemma}
    Let $X$ be a proper variety. Then $\operatorname{Big}(X) \subset N^1(X)$ is an open subcone. 
\end{lemma}
\begin{proof}
    Let $B$ be a big $\bb R$-divisor and write $B = \sum_i a_i B_i$ for non-negative real numbers $a_i > 0$ and big integral divisors $B_i$. It suffices to show for any integral divisor $D$ there is $\epsilon>0$ such that $B + \epsilon D$ is big. Since the sum of big $\bb R$-divisors is big by definition, we may assume $B$ is a big $\bb Q$-divisor and then the claim follows by Lemma \ref{lem: big is open}. 
\end{proof}

Here are some quick general observations on the shape of the $\otimes$-generating cone.
\begin{lemma}
\label{lemma-basicsoftensorampleRdivisors}
    Let $X$ be a proper variety over a field $k$.
    \begin{enumerate}
        \item Writing $\tensgen_+(X)= \tensgen(X) \cap \operatorname{Big}(X)$ and $\tensgen(X)_-= \tensgen(X)\cap (-\operatorname{Big}(X))$, we have the decomposition
        \[
        \tensgen(X) = \tensgen_+(X) \cup \tensgen_-(X).
        \]
        \item For a closed subvariety $Z \subset X$, let $$\operatorname{Big}_{Z,+}(X)\subset \operatorname{Big}(X) \quad (resp. \operatorname{Big}_{Z,-}(X)\subset \operatorname{Big}(X))$$ denote the cone of big $\bb R$-divisors whose restriction to $Z$ is big (resp. anti-big). Then writing
        \[
        \tensgen_\sigma(X) := \bigcap_{Z\subset X} \operatorname{Big}_{Z,\sigma(Z)}(X)
        \]
        for each $\sigma \in \Sigma$, where $\Sigma$ denotes the set of maps from the set of closed subvarieties $Z\subset X$ to $\{+,-\}$ that assign $+$ if $\dim Z =0$, we have the following decomposition into convex cones:
        \[
        \tensgen_+(X) = \bigsqcup_{\sigma \in \Sigma} \tensgen_\sigma(X) \subset \operatorname{Big}(X).
        \]
        \item 
        We have $$\tensgen (X) \cap \operatorname{Nef}(X) = \bigcap_{Z\subset X} \operatorname{Big}_{Z,+}(X) = \operatorname{Amp}(X).$$ 
        In particular, $\tensgen(X) \cap  \operatorname{\partial Nef}(X) = \emp$. \qedhere 
    \end{enumerate}
\end{lemma}
\begin{proof}
    Part (i) and (ii) follow by definition, where the disjointness in part (ii) follows from the fact that $\operatorname{Big}(Z) \cap (- \operatorname{Big}(Z)) = \emp$ for any proper variety $Z$ of positive dimension. Part (iii) follows by the same strategy as Lemma \ref{lemma-nefandtensorample}. First, if $[D] \in \tensgen_+(X) \cap \operatorname{Nef}(X)$ for an $\bb R$-divisor $D$, then for any closed subvariety $Z \subset X$, $[D]|_Z$ is big since otherwise $-[D]|_Z$ is big and there is an integral curve $C \subset Z$ such that $(-D|_Z \cdot C)>0$ (by taking a curve that intersects positively with each big integral divisor appearing in a representative of the class $-[D]|_Z$ using Lemma \ref{lemma-bigequiv} condition (ii)), which is absurd as $[D]|_Z$ is nef. Thus, we have
   \[
   \tensgen(X) \cap \operatorname{Nef}(X) = \bigcap_{Z\subset X} \operatorname{Big}_{Z,+}(X)
   \]
   as the other inclusion is trivial. To see the other equality, since the restriction of a $\otimes$-generating and nef $\bb R$-divisor to a closed subvariety is $\otimes$-generating and nef, it suffices to show for a $\otimes$-generating and nef (and hence big) $\bb R$-divisor $A$ on a proper variety $Z$, we have $(A^{\dim Z})>0$ by the Nakai-Moishezon criterion for $\bb R$-divisors on a proper variety over an algebraically closed field (\cite{fujino2023nakai}*{Theorem 1.3}). Indeed, by Kodaira's lemma (Lemma \ref{lem: Kodaira}), after replacing $A$ with its multiple, we can write $A = E + F$, where $E$ is an integral effective divisor and $F$ is an effective $\bb R$-divisor. Then, by Kleiman's theorem (\cite{Kollarrationalcurves}*{Chapter VI, Theorem 2.17}), we have $(A^{\dim Z}) = ((E+F)\cdot A^{\dim Z -1}) \geq (E \cdot A^{\dim Z -1})$ and we are done by induction on $\dim Z$ as $[A]|_E$ is $\otimes$-generating and nef again. 
\end{proof}

\begin{question}
    For which $\sigma \in \Sigma$, is $\tensgen_\sigma(X)$ open/non-empty? For which $\sigma \in \Sigma$, there exists $Z \subset X$ such that $\tensgen_\sigma = \operatorname{Big}_{Z,\pm}(X)$? 
\end{question}
\begin{example}
    The cone $\tensgen_+(X)$ is not necessarily convex. Indeed, let $X$ be a projective surface and suppose we have $D \in \tensgen_+(X) \setminus \operatorname{Amp}(X)$ (see Corollary \ref{corollary-tesnampnotamponsurface} for examples). Then, there is an integral curve $C \subset X$ such that $(C\cdot D) < 0$. Now, for an ample divisor $A$ on $X$, we have
    \[
    \left(D - \frac{(C\cdot D)}{(C\cdot A)}A\right )\cdot C = 0.
    \]
    Thus, we get a non-$\otimes$-generating $\bb R$-divisor as a positive linear combination of two big and $\otimes$-generating $\bb R$-divisors.
\end{example}

We will completely determine the $\otimes$-generating cone of ruled surfaces later, see Example \ref{example-ruledsurfaces}. In some cases, we will see that the $\otimes$-generating cone of a ruled surface is an open convex cone in $\mathbb{R}^2$ minus a ray, which gives a concrete example of non-convexity. 

\section{Examples}
\label{section-examples}

\subsection{\texorpdfstring{$\otimes$-}{Tensor-}generation on curves}
\label{subsection-curves}

\begin{prop}\label{lem: tensor ample on curves}
    Let $C$ be an integral projective curve over a field and let $\ecal L$ be a line bundle on $C$. Then the following are equivalent:
    \begin{enumerate}
        \item $\ecal L$ is $\otimes$-generating,
        \item $\deg \ecal L\neq 0$,
        \item $\ecal L$ is big or anti-big.
        \item $\ecal L$ is ample or anti-ample \qedhere
    \end{enumerate} 
\end{prop}

If $C$ is a proper scheme over a field which is purely of dimension 1 then the degree of a line bundle $\ecal{L}$ on $C$ is $\chi(C, \ecal{L}) - \chi(C, \ecal{O}_C)$. See \cite[\href{https://stacks.math.columbia.edu/tag/0AYQ}{Tag 0AYQ}]{stacks-project} as a reference. Note that the Proposition above is not true if $C$ is not assumed integral by Example \ref{example-reducibleconic}.

\begin{proof}
    This follows from Theorem \ref{ref-theoremmain} since a line bundle on a projective curve is ample if and only if it is big if and only if it has positive degree.
\end{proof}

Amplitude of a line bundle on the fibers of a proper, finitely presented morphism is an open condition on the base by \cite{EGA4Part3}. This is not the case for $\otimes$-generation.

\begin{example}
    Let $R$ be a complete DVR and let $\ecal{C} / R$ be a flat family of conics in $\mathbb{P}^2_{R}$ with reducible special fiber $\ecal{C}_0$ and smooth generic fiber. The line bundle on $\ecal{C}_0$ which is $\ecal{O}(1)$ on one irreducible component and $\ecal{O}(-1)$ on the other lifts to a line bundle $\ecal{L}$ on $\ecal{C}$ (there is a formal lift since $H^2(\ecal{C}_0, \ecal{O}_{\ecal{C}_0}) = 0$ and it algebraizes by Grothendieck's Existence Theorem). Since the degree of a line bundle on a proper, flat family of 1-dimensional schemes is constant, we see that the restriction of $\ecal{L}$ to the generic fiber is trivial (the only line bundle of degree zero on a smooth conic over a field). Then $\ecal{L}$ is $\otimes$-generating when restricted to the geometric special fiber but not $\otimes$-generating when restricted to the geometric generic fiber.
\end{example}

\subsection{\texorpdfstring{$\otimes$-}{Tensor-}generation on surfaces}
\label{subsection-surfaces}
Similarly to the case of curves, we can give a simpler characterization for $\otimes$-generating line bundles on surfaces. 
\begin{lemma}
\label{lemma-tensoramplenessonsurface}
    Let $X$ be a smooth projective surface over a field. Let $\ecal{L}$ be a big line bundle on $X$. Then $(\ecal{L}. C) > 0$ for every irreducible curve $C \subset X$ with $C^2 \geq 0$. In particular, $\ecal{L}$ is:
    \begin{enumerate}
        \item Ample if and only if $(\ecal{L}.C)>0$ for every irreducible curve $C \subset X$ with $C^2 < 0$.
        \item $\otimes$-generating if and only if $(\ecal{L}. C) \neq 0$ for every irreducible curve $C \subset X$ with $C^2 < 0$. \qedhere
    \end{enumerate}
\end{lemma}

\begin{proof}
    If $C \subset X$ is an irreducible curve with $C^2 \geq 0$, then $C$ is a nef divisor. Let $\ecal{M}$ be an ample line bundle on $X$ and choose $n \gg 0$ so that $\ecal{L}^{\otimes n} \otimes \ecal{M}^{-1}$ is effective (Kodaira's Lemma). Then since $C$ is nef we have
    $$
    0 \leq ((\ecal{L}^{\otimes n} \otimes \ecal{M}^{-1} ). C) = n (\ecal{L}.C) - (\ecal{M}.C) < n (\ecal{L}.C)
    $$
    hence $(\ecal{L}.C) > 0$, as needed. Then (i) follows from the Nakai--Moishezon criterion ($\ecal{L}^2> 0$ holds because $\ecal{L}$ is big and nef in this case) and (ii) follows from Theorem \ref{ref-theoremmain}.
\end{proof}

Using the corresponding results for $\bb R$-divisors, the same proof gives a detailed description of the cone $\tensgen_+(X)$, see Definition \ref{defn-tensamplecone}.

\begin{lemma}\label{lem: tensor ample cone for surfaces}
    Let $X$ be a smooth projective surface over an algebraically closed field. We have
    $$
    \tensgen_+(X) = \operatorname{Big}(X) \setminus \bigcup \{C^\perp : C \subset X \text{ an integral curve with } C^2 < 0\},
    $$
    and this union is countable.
\end{lemma}

\begin{proof}
    The inclusion $\subset$ is clear. For the reverse inclusion, if $D$ is a big $\mathbb{R}$-divisor on $X$ and $C \subset X$ is an integral curve with $C^2 \geq 0$, then by Kodaira's Lemma we can write $D = A + E$ where $A$ is an ample $\mathbb{R}$-divisor and $E$ is an effective $\mathbb{R}$-divisor. Then $(D \cdot C) = (A \cdot C) + (E \cdot C) \geq (A \cdot C) > 0$ since $C$ is nef and $A$ is ample. The last part follows because there are only countably many integral curves on $X$ of negative self square because any two represent distinct elements of the Neron--Severi group. 
\end{proof}
In particular, if there are only finitely many negative self-intersection curves, then $\tensgen_+(X)$ is open. 
\begin{corollary}
\label{corollary-tesnampnotamponsurface}
    Let $X$ be a smooth projective surface over a field. The following are equivalent.
    \begin{enumerate}
        \item There exists a big line bundle $\ecal{L}$ on $X$ which is $\otimes$-generating but not ample. 
        \item There exists an integral curve $C \subset X$ such that $C^2 < 0$. \qedhere
    \end{enumerate}
\end{corollary}

\begin{proof}
The direction (i) $\implies$ (ii) follows directly from Lemma \ref{lemma-tensoramplenessonsurface}. For the other direction, let $\ecal{L}$ be a very ample line bundle on $X$. Choose an integer $n> 0$ such that $-n C^2> (\ecal{L}. C) > 0$. Then the line bundle $\ecal{L} \otimes \ecal{O}(n C)$ is big being a tensor product of an ample and an effective line bundle. For every $x \in X \setminus C$ there is a section $s \in \Gamma(X, \ecal{L})$ such that $X_s$ is an affine open neighborhood of $x$ and then the section $s \otimes t \in \Gamma(X, \ecal{L}\otimes \ecal{O}(n\cdot C))$ where $t  \in \Gamma(X, \ecal{O}(n C))$ is the section cutting out $n C$ satisfies: $X_{s \otimes t} \subset X_s$ is a quasi-affine open neighborhood of $x$. Also, we have 
\begin{equation}
\label{equn-negativeintersection}
(\ecal{L} \otimes \ecal{O}(nC) . C) =  (\ecal{L}.C) + n C^2 < 0.
\end{equation}
We conclude by Lemma \ref{lemma-usefulcriterion} that $\ecal{L} \otimes \ecal{O}(n C)$ is $\otimes$-generating and by (\ref{equn-negativeintersection}) it is not ample. 
\end{proof}

If we want to check $\tens$-generation of canonical bundles, then these criteria can be simplified. 

\begin{lemma}\label{lem: need to be (-2)}
    Let $X$ be a smooth projective surface with (anti-)big canonical bundle over an algebraically closed field $k$. Let $C \subset X$ be an integral curve. Then $K_X\cdot C = 0$ implies $C$ is a $(-2)$-curve, i.e.,  $C \cong \mathbb{P}^1_k$ and $C^2 = -2$. 
\end{lemma}
\begin{proof}
    The claim follows by \cite[Claim 3.9]{A_nsings} when $\omega_S$ is big and by \cite{BPtoric}*{\href{https://arxiv.org/pdf/1010.1717}{Lemma 8}} when $\omega_S$ is anti-big. Alternatively, by Lemma \ref{lemma-tensoramplenessonsurface}, we see $C^2 < 0$.  By the adjunction formula, we have
    $$
    \operatorname{deg}(\omega_C) = K_X.C + C^2 < 0
    $$
    and $\operatorname{deg}(\omega_C) = 2 p_a(C) - 2$ where $p_a(C) = \operatorname{dim}_k H^1(C, \ecal{O}_C)$ is the arithmetic genus of $C$. Since $p_a(C) \geq 0$ we see $p_a(C) = 0$ and $C^2 = -2$. We get that $C \cong \mathbb{P}^1_k$ because the arithmetic genus of a singular curve is greater than the genus of its normalization.
\end{proof}
\begin{lemma}\label{lem: tensor ample canonical on a surface}
    Let $X$ be a smooth projective surface over an algebraically closed field. Then, $\omega_X$ is $\otimes$-generating if and only if $\omega_X$ is big or anti-big and $X$ contains no $(-2)$-curve. 
\end{lemma}
\begin{proof}
    If $\omega_X$ is $\otimes$-generating, then $\omega_X$ is big or anti-big and for any integral curve $C \subset X$, we have 
    \[
   0 \neq C\cdot K_X = 2g(C) - 2 - C^2
    \]
    by the adjunction formula. In particular, $X$ contains no $(-2)$-curve. 
    For the converse, note Corollary 8 in the note implies there is an integral curve $C \subset X$ such that $\omega_X|_C$ is neither big nor anti-big and hence $K_X \cdot C = 0$ by Lemma \ref{lem: tensor ample on curves}. Now, Lemma \ref{lem: need to be (-2)} shows $C$ is a $(-2)$-curve. 
\end{proof}
\begin{remark}
    The only if direction works over any field. 
\end{remark}
Now, let us apply these criteria to some specific classes of examples.
\paragraph{Toric surfaces} In this paragraph, we work over an algebraically closed field. 
Let us first give a general criterion of $\tens$-generation of the canonical bundle on a Gorenstein toric variety.
\begin{notation} Let $X = X_\Sigma$ be a (normal) toric variety given by a fan $\Sigma$ and let $T \subset X$ be the corresponding torus. By \cite{cox2011toric}*{Theorem 8.2.3}, the canonical sheaf $\omega_X$ on a toric variety $X$ can be written as $\omega_X = \ecal O_X(- \sum_{\rho \in \Sigma(1)} D_\rho)$, where $\Sigma(1)$ denotes the set of rays of $\Sigma$ and $D_\rho$ denotes the corresponding torus-invariant prime Weil divisor. Define the \emph{toric boundary divisor} to be $B := \bigcup_{\rho \in \Sigma(1)} D_\rho$. Note $X \setminus B = T$ since (the $T$-orbit of) any closed point $x \not \in T$ is contained in one of the torus-invariant prime Weil divisor for example by the Cone-Orbit correspondence \cite{cox2011toric}*{Theorem 3.2.6.}. See also \cite{cox2011toric}*{Section 4.1.} for discussions of basic notions. 
\end{notation}
The following generalizes a well-known fact on the bigness of the anticanonical bundle on a smooth projective toric variety (e.g. \cite{BPtoric}*{Lemma 11}) to not necessarily proper cases using our definition. 
\begin{lemma}\label{lemma-torusanticanonicalbig}
    Let $X$ be a Gorenstein toric variety. Then the canonical bundle $\omega_X$ is anti-big. 
\end{lemma}
\begin{proof}
By \cite{cox2011toric}*{Proposition 8.2.12.}, the canonical sheaf $\omega_X = \ecal O_X(- \sum_{\rho \in \Sigma(1)} D_\rho)$ is a line bundle. Now, the section of $\omega_X^{-1}$ corresponding to the effective Weil divisor $\sum _{\rho \in \Sigma(1)} D_\rho$ vanishes precisely along $B = X \setminus T$, so we are done by Lemma \ref{lemma-bigequiv} condition (ii) as $T$ is affine.      
\end{proof}
\begin{prop}
    Let $X$ be a Gorenstein toric variety and let $B \subset X$ denote the toric boundary divisor. Then the following are equivalent. 
    \begin{enumerate}
        \item The canonical bundle $\omega_X$ is $\otimes$-generating.
        \item The restriction $\omega_X|_B$ is $\otimes$-generating. 
        \item The restriction $\omega_X|_{D_\rho}$ is $\otimes$-generating for every $\rho \in \Sigma(1)$. \qedhere
    \end{enumerate}
\end{prop}
\begin{proof}
(i) $\implies$ (ii) is clear and (ii) $\Leftrightarrow$ (iii) follows by Lemma \ref{lemma-irreduciblecomp} as $B = \bigcup_{\rho \in \Sigma(1)} D_\rho$. (ii) $\implies$ (i) follows by taking a section of $\omega_X^{\tens -1}$ vanishing precisely along $B = X \setminus T$ as in the proof of Lemma \ref{lemma-torusanticanonicalbig}, and by applying Lemma \ref{lemma-usefulcriterion} as $T$ is affine. 
\end{proof}
\begin{remark}
    By the same proofs as above, the preceding lemma holds for a line bundle of the form $\ecal O_X(nB+E)$ for any torus-invariant effective Cartier divisor $E \subset X$ and any nonnegative integer $n \geq 0$, which is necessarily a big line bundle.
\end{remark}
As a direct corollary for surfaces, we obtain the following.
\begin{corollary}
    Let $X$ be a Gorenstein toric surface. Then, $\omega_X$ is $\otimes$-generating if and only if $\omega_X|_{C}$ is $\otimes$-generating for every torus-invariant integral curve $C \subset B$. 
\end{corollary}
When a toric surface $X$ is moreover smooth and projective, we can apply Lemma \ref{lem: tensor ample canonical on a surface} to get a complete classification of such surfaces with $\otimes$-generating canonical bundle. Here, recall the toric boundary divisor $B$ of a smooth projective toric surface consists of a cycle of rational curves $C_i$ and the configuration of $C_i$ together with their self-intersection numbers $a_i$ uniquely determines a smooth projective toric surface up to isomorphism. See \cite{oda1988convex}*{Corollary 1.29} for details. 
\begin{corollary}
    Let $X$ be a smooth projective toric surface corresponding to the data $\{(C_i,a_i)\}$. Then, $\omega_X$ is $\otimes$-generating if and only if $a_i \neq -2$ for all $i$. 
\end{corollary}
\begin{proof}
    Since $\omega_X$ is anti-big, $\omega_X$ is $\otimes$-generating if and only if $X$ has no $(-2)$-curve by Lemma \ref{lem: tensor ample canonical on a surface}. Now, since any integral curve $C\subset X$ with $C^2 < 0$ must be contained in the toric boundary divisor $B$, we are done by the description of $B$ as above. 
\end{proof}
\paragraph{Ruled surfaces}
Using Lemma \ref{lem: tensor ample cone for surfaces}, we can give a fully explicit description of $\tensgen_+(X)$ when $X$ is a geometrically ruled surface over a smooth projective curve over the complex numbers.
\begin{example}
\label{example-ruledsurfaces}
Let $\pi:X = \bb P_X(\ecal E) \to C$ be the projective bundle of a vector bundle $\ecal E$ of rank $2$ and degree $e$ over a smooth projective curve $C$ of genus $g$ over $\bb C$.
Let $f,\xi \in N^1(X)$ denote the numerical class of a fiber of $\pi$ and that of $\ecal O_{\bb P(\ecal E)}(1)$, respectively. Recall that $N^1(X)$ is spanned by $f$ and $\xi$ and that we have 
\[
f^2 = 0,\quad f\cdot \xi = 1,\quad \xi^2 = e.
\]
 In the sequel, we use the descriptions of $\operatorname{PsEff}(X)$ and $\operatorname{Nef}(X)$ given by Fulger \cite{fulger2011cones}.
    \begin{enumerate}
        \item Suppose $\ecal E$ is unstable and take a destabilizing quotient line bundle $\ecal E \surj \ecal Q$, which gives a section $C_0 = \bb P_X(\ecal Q) \inj X$. Set $d:= \deg (\ecal Q) < \mu(\ecal E) = e/2$. By Miyaoka's result (\cite{miyaoka1987chern}; see \cite{fulger2011cones}*{Lemma 2.1}), the nef cone is spanned as $\operatorname{Nef}(X) = \bra{\xi - d f, f}$ and by \cite{fulger2011cones}*{Lemma 2.3}, the pseudoeffective cone is spanned as $\operatorname{PsEff}(X) = \bra{[C_0],f}$. Here, note we have
        \[
        [C_0] = \xi + (d-e) f,\quad [K_X] = - 2\xi  + (2g - 2 + e)f.
        \]
        Since $C_0 \subset X$ is the unique integral curve in $X$ with negative self-intersection $C_0^2 = 2d - e < 0$, Lemma \ref{lem: tensor ample canonical on a surface} shows  
        \[
        \tensgen_+(X) = \operatorname{Big}(X) \setminus C_0^\perp = \operatorname{Big}_{C_0,+}(X) \sqcup \operatorname{Big}_{C_0,-}(X),
        \]
        where we can easily see $C_0^\perp = \bb R_+ (\xi - d f)$ and $\operatorname{Big}_{C_0,+}(X) = \operatorname{Amp}(X)$. Therefore, $\tensgen(X)_+$ looks as follows (when $d < e \leq 0$ and $d - e < 0 < 1 - g -e/2 < -d$): 
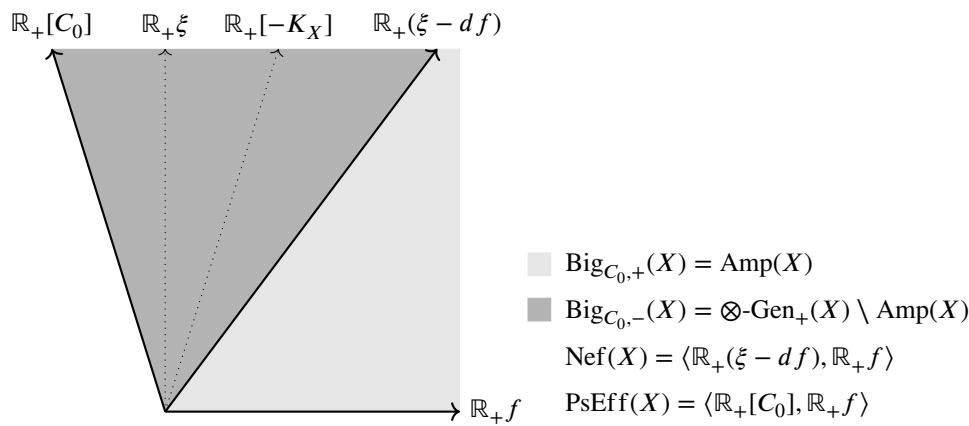
        \begin{center}
\begin{tikzpicture}[scale=3]
  \draw[thick, ->] (0,0) -- (1.3,0) node[right] {$\bb R_+ f$};
  \draw[thick, ->] (0,0) -- (-0.5,1.6) node[above] {$\bb R_+ [C_0]$};
  \draw[thick, ->] (0,0) -- (1.2,1.6) node[above] {$\bb R_+ (\xi - d f)$};
  \draw[dotted, ->] (0,0) -- (0,1.6) node[above] {$\bb R_+ \xi$};
  \draw[dotted, ->] (0,0) -- (0.5,1.6) node[above] {$\bb R_+ [-K_X]$};
  \fill[opacity=0.1] (0,0) -- (1.3,0) -- (1.3,1.6) -- (1.2,1.6) -- cycle;
  \fill[opacity=0.3](0,0) -- (-0.5,1.6) -- (1.2, 1.6) -- cycle;

  \begin{scope}[shift={(1.6,0.6)}]
    \draw[fill = black, opacity=0.1] (0,0) rectangle (0.1,0.1);
    \node[right=2pt] at (0.1, 0.035) {$\operatorname{Big}_{C_0,+}(X) = \operatorname{Amp}(X)$};
    \draw[fill = black, opacity=0.3] (0,-0.2) rectangle (0.1,-0.1);
    \node[right=2pt] at (0.1,-0.17) {$\operatorname{Big}_{C_0,-}(X) = \tensgen_+(X) \setminus \operatorname{Amp}(X)$};
    \node[right=2pt] at (0.1,-0.37) {$\operatorname{Nef}(X) = \bra{\bb R_+ (\xi - d f), \bb R_+ f}$};
    \node[right=2pt] at (0.1,-0.57) {$\operatorname{PsEff}(X) = \bra{\bb R_+[C_0],\bb R_+f}$};
  \end{scope}
\end{tikzpicture}
\end{center}
Now, by Lemma \ref{lemma-basicsoftensorampleRdivisors} the $\otimes$-generation of the canonical divisor $K_X$ can be checked by observing whether or not the ray $\bb R_+ [- K_X] = \bb R_+(\xi +(1-g-e/2)f)$ goes through $\tensgen(X)_+$ as $K_X$ cannot be big. 
        
If $g = 0$, then we may assume $\ecal E = \ecal O_C \oplus \ecal O_C(d)$ with $e = d <0$ (i.e., $X$ is the Hirzebruch surface $\bb F_{{-}d}$) and
        \[
        \text{$-K_X$ is} \quad
        \begin{cases}
            \text{$\otimes$-generating but not ample} & \text{if $d \leq -3$} \\
            \text{big but not $\otimes$-generating} & \text{if $d = -2$} \\
            \text{ample} & \text{if $d = -1$}
        \end{cases}
        \]
        If $g = 1$, then we may assume $\ecal E = \ecal O_C \oplus \ecal O_C(d)$ with $e = d <0$ (as indecomposable sheaves on an elliptic curve are semistable) and we see $-K_X$ is $\otimes$-generating but not ample for any $d$. 

        If $g\geq 2$, then
        \[
        \text{$- K_X$ is} \quad
        \begin{cases}
           \text{not pseudo-effective} & \text{if $d > 1 - g + e/2$} \\
            \text{pseudo-effective but not big} & \text{if $d = 1 - g + e/2$} \\
            \text{$\otimes$-generating but not ample} & \text{if $d < 1 - g + e/2$} 
        \end{cases}
        \]
See also Lemma \ref{lem: need to be (-2)} for a more general claim on $\tens$-generation of (anti-)big canonical bundle. 
        \item Suppose $\ecal E$ is semistable. Then by \cite{fulger2011cones}*{Lemma 2.2, Lemma 3.3} we have $\operatorname{PsEff}(X) = \operatorname{Nef}(X) = \bra{\bb R_+(\xi - \frac{e}{2} f),\bb R_+f}$ and hence $$\operatorname{Big}(X) = \operatorname{Amp}(X) = \tensgen_+(X) = \operatorname{int}(\bra{\bb R_+(\xi - \frac{e}{2} f),\bb R_+f}). $$
        Hence, $K_X$ is $\otimes$-generating if and only if $-e/2 < 1- g-e/2$, i.e., $g = 0$. \qedhere
    \end{enumerate}
\end{example}
\begin{remark}
    From the computations above, we see that if a $\bb P^1$-bundle $\bb P_E(\ecal E) \to E$ over an elliptic curve $E$ has a non-isomorphic Fourier--Mukai partner, then $\ecal E$ is a semi-stable vector bundle, which is a weaker version of Uehara's result \cite{Ellipticruledsurface}*{Theorem 1.1}.  
\end{remark}

\paragraph{Blow-ups of $\bb P^2$}
{In this paragraph we work over a fixed algebraically closed field which we suppress from notation, and a point means a closed point unless otherwise stated.} We fix the following notation. 
\begin{notation}
     Let $X$ be a blow-up of $\bb P^2$ at $r$ distinct points $p_1,\dots, p_r \in \bb P^2$ (not allowing infinitely near points). Let $H$ be the pull-back of a hyperplane avoiding all $p_i$'s in $\bb P^2$ and let $E_i$ denote the exceptional divisor at each point $p_i$ for $i = 1,\dots, r$. Then $K_X= - 3H + \sum_{i=1}^r E_i$. Note 
    \[
    K_X^2 = 9 - r. 
    \]
    Also, for a rational number $q \in \bb Q$, consider an $\bb Q$-divisor $D_q = qH - \sum_{i = 1}^r E_i$. Here, note $D_3 = - K_X$. 
\end{notation}
The following long-standing and widely open conjecture of Nagata (\cite{nagata195914}) indicates $\otimes$-generation (as well as positivity in general) of line bundles on a blow-up of $\bb P^2$ can be quite subtle. See also \cite[2.2]{VariationsonNagatasconjecture}. 
\begin{conjecture}[Nagata]
    Suppose $p_1,\dots,p_r$ are $r$ very general points in $\bb P^2$. Then, for any integral curve $C \subset \bb P^2$, we have
    \begin{equation}
    \label{equn-nagatainequ}
    \deg C > \frac{1}{\sqrt{r}} \sum_{i = 1}^r \operatorname{mult}_{p_i}(C),
    \end{equation}
    where $\operatorname{mult}_{p_i}(C)$ denotes the multiplicity of $C$ at $p_i$. 
    Equivalently, the linear system $|d H - \sum_{i = 1} m_iE_i|$ has no integral curve for any positive integers $d$ and $m_i$'s satisfying $\sqrt{r} \leq  \frac{1}{d}\sum_{i=1}^rm_i.$
\end{conjecture}
\begin{remark}
    Nagata has shown the conjecture when $r$ is a perfect square and when $r$ is not a perfect square, the inequality (\ref{equn-nagatainequ}) is equivalent to the same inequality but with $>$ replaced by $\geq$. 
\end{remark}
The following equivalent formulations of Nagata's conjecture in terms of $\bb Q$-divisors $D_q$ are well-known (except for the condition (iv)) (cf. e.g. \cite{biran1999constructing} and \cite{lazarsfeld2017positivity}*{Remark 5.1.14}), but let us include proofs of equivalences for completeness since we could not find proofs explicitly written down in literature. It is worth emphasizing that the condition (iv) can be stated purely in the language of generation of derived categories by definition, which seems to be a new perspective.
\begin{lemma}\label{lemma-equiv Nagata}
    Suppose $p_1,\dots,p_r$ are $r$ very general points in $\bb P^2$. The following are equivalent. 
    \begin{enumerate}
        \item Nagata's conjecture holds {for curves through $p_1, \dots , p_r$.}
        \item The $\bb R$-divisor $N:= \sqrt{r}H - \sum_{i=1}^rE_i$ is {strictly} nef.
        \item The $\bb Q$-divisor $D_q$ is ample for any rational number $q > \sqrt{r}$. 
        \item The $\bb Q$-divisor $D_q$ is $\otimes$-generating for any rational number $q > \sqrt{r}$.
        \item The $\bb Q$-divisor $D_q$ is nef for any rational number $q > \sqrt{r}$.
        \qedhere 
    \end{enumerate}
\end{lemma}
\begin{proof}
To see (i) $\implies$ (ii), assume the $\bb R$-divisor $N$ is not strictly nef and take an integral curve $C \subset X$ such that $(N\cdot C) \leq  0$. Since $C \neq E_i$, we can write $C = dH - \sum_{i=1}^r m_i E_i$ with a positive integer $d$ and nonnegative integers $m_i$. In particular, $(H \cdot C) = d > 0$. Therefore, we can take a rational number $q \geq \sqrt{r}$ so that $(D_q \cdot C) = qd - \sum_{i = 1}^r m_i = 0$, which contradicts with Nagata's conjecture since then we have $\sqrt{r} \leq q = \frac{1}{d}\sum_{i = 1}^r m_i$. To see that (ii) $\implies$ (iii), we have for a rational number $q > \sqrt{r}$ that $D_q = N + (q - \sqrt{r})H$ and $(q - \sqrt{r})H$ is nef so $D_q$ is strictly nef. 
Then we are done as $D_q^2 = q^2 -r > 0$ for any rational number $q > \sqrt{r}$. (iii) $\implies$ (iv) is trivial. For (iv) $\implies$ (i), note for any $d',m_1',\dots,m_r' \in \bb Z$, we have $D_q \cdot (d' H - \sum m_i' E_i) =  qd' - \sum_{i = 1}^r m_i'.$
    Thus, if the $\bb Q$-divisor $D_q$ is $\otimes$-generating for a rational number $q \in \bb Q$, then the linear system $|d'H - \sum_{i = 1}^r m_i' E_i|$ has no integral curve for any $d',m_1',\dots,m_r' \in \bb Z$ satisfying $q =\frac{1}{d'} \sum_{i = 1}^r m_i'$ by Lemma \ref{lem: tensor ample on curves}. Thus, if $D_q$ is $\otimes$-generating for any rational number $q \geq\sqrt{r}$, then Nagata's conjecture follows. If $r$ is not a perfect square, then this follows from the assumption (iv). If $r$ is a perfect square then Nagata's conjecture holds, see \cite[Section 3]{nagata195914} or \cite[2.2]{VariationsonNagatasconjecture}. Finally, (iii) $\implies$ (v) is clear and for (v) $\implies$ (iii), note for any rational number $q > \sqrt{r}$, we can take a rational number $q'$ with $q >q'>\sqrt{r}$. If $D_{q'}$ is strictly nef, then $D_q$ is ample as in the proof of (ii) $\implies$ (iii) and if $D_{q'}$ is nef with some integral curve $C \subset X$ satisfying $(D_{q'} \cdot C) = 0$, then $(H \cdot C) > 0$ as in the proof of (i) $\implies$ (ii) and hence $D_q$ is strictly nef with $D_q^2 > 0$, so $D_q$ is ample. 
\end{proof}
\begin{remark}\label{remark-current Nagata}
    To the authors' knowledge, history and current status of positivity of $\bb Q$-divisors $D_q$ seem to be as follows. In \cite{xu1995divisors} and \cite{kuchle1996ample}, Xu and K\"uchle independently show that for $r \geq 3$, an integral divisor $D_q$ is ample for any integer $q > \sqrt{r}$. Then, in \cite{biran1999constructing}*{Corollary 2.1.B}, Biran shows the $\bb Q$-divisor $D_{d/2}$ is nef for any integer $d$ satisfying $d/2 \geq \sqrt{r}$. Later, in \cite{harbourne2003Seshconst}*{Corollary I.5}, Harbourne further shows for $r \geq 2$ the $\bb Q$-divisor $D_{d/e}$ is nef for any positive integers $e \leq \sqrt{r}$ and $d$ satisfying $d/e > \sqrt{r}$. Another active approach is use of Seshadri constants. See \cite{lazarsfeld2017positivity}*{Section 5.1} or \cite{bauer2009primer} for surveys and further references. 
\end{remark}

Now, let us turn to $\tens$-generation of the (anti)canonical divisor $D_3 = -K_X$ as a special case. Recall if $r<9$ and if we blow-up $r$ general points, then $X$ is a del Pezzo surface and in particular $\omega_X$ is $\otimes$-generating, so for general blow-ups, we focus on cases where $r\geq 9$.  
\begin{prop} If the linear system $\l|dH - \sum_{i=1}^r a_iE_i \r|$ contains an integral curve for some positive integer $d$ and nonnegative integers $a_1,\dots,a_r$ satisfying $\sum_{i=1}^r a_i = 3d$, then $\omega_X$ is not $\otimes$-generating. In particular, if $r\geq 9$ and if $p_1,\dots,p_r$ are $r$ general points in $\bb P^2$, then $\omega_X$ is not $\otimes$-generating.
\end{prop}
\begin{proof}
    The first claim follows by Lemma \ref{lem: tensor ample on curves} as above noting $K_X\cdot (dH - \sum_{i=1}^r a_iE_i) = - 3d + \sum_{i=1}^r a_i$. For the latter claim, note $\l|dH - \sum_{i=1}^r a_iE_i \r|$ has an integral curve if and only if there is an integral curve of degree $d$ that goes through $p_i$ with multiplicity $a_i$. Now, by dimension count there is an integral cubic plane curve going through $9$ general points with multiplicity $1$ at each point (and avoiding the rest of the general points when $r>9$), so $|3H - \sum_{i=1}^9E_i|$ contains an integral curve and we are done. 
\end{proof}
\begin{remark}
    Unless $X$ is a rational elliptic surface,
    $X$ has no non-isomorphic Fourier--Mukai partners \cite[Theorem 1.3]{Kawamata2002DEquivalenceAK}. In particular, most blow-ups of $\bb P^2$ have no non-isomorphic Fourier--Mukai partners, see \cite{Krahblowups}.
\end{remark}
The preceding results mean we need special configurations of $r$ points to have a $\otimes$-generating canonical line bundle when $r\geq 9$. 
\begin{example}
    Take a line $l \subset \bb P^2$ and assume the points $p_1, \dots, p_r$ lie on $l$. Then its strict transform $\tilde l$ lies in $|H - \sum_{i=1}^rE_i|$. Therefore, we have
    \[
    3 \tilde l + \sum_{i=1}^r 2E_i \in |-K_X| = |3H - \sum_{i=1}^r E_i|.
    \]
    Note if $r\neq 3$, then we have $-K_X\cdot \tilde l = 3 - r \neq 0$ (and $-K_X\cdot E_i = 1$ for any $r$). Thus, if $r\neq 3$, then by taking the section $s$ of $\omega^{\tens -1}_X$ corresponding to $3 \tilde l + \sum_{i=1}^r 2E_i$, we see $\omega_X$ is $\otimes$-generating by 
    Lemma \ref{lemma-usefulcriterion} as $X_s$ is isomorphic to $\bb P^2\setminus l$ and hence is affine. Note if $r \geq 9$, then $(-K_X)^2 = 9 - r \leq 0$, so $-K_X$ is big and $\otimes$-generating, but not ample so not nef either by Lemma \ref{lemma-nefandtensorample}. 

    Similarly, take a conic $C \subset \bb P^2$ and assume the points $p_1, \dots, p_r$ lie on $C$. Then its strict transform $\tilde{C}$ lies in $|2H - \sum_{i = 1}^r E_i|$. Therefore, we have 
    \[
    3 \tilde{C} + \sum_{i= 1}^r E_i \in |-2K_X| = |6H - \sum_{i=1}^r 2E_i|
    \]
    and if $r\neq 6$, then $-K_X \cdot \tilde{C} = 6 - r \neq 0$. Thus, taking the section of $\omega_X^{\tens -2}$ corresponding to $3 \tilde{C} + \sum_{i= 1}^r E_i$, we see $\omega_X$ is $\otimes$-generating as above if $r\neq 6$. Again, if $r \geq 9$, then $-K_X$ is big and $\otimes$-generating, but not nef. 

    Note these arguments no longer work for plane curves of higher degree. 
\end{example}
\subsection{Higher dimensional blowups in points}
The following shows $\otimes$-generating line bundles naturally arise from blow-ups. 
        \begin{prop}
        \label{prop-blowuptensamp}
            Let $X$ be a quasi-projective variety over a field and let $\pi:Y \to X$ be the blow-up at finitely many closed points $x_1,\dots,x_n  \in X$. Let $\ecal M$ be an ample line bundle on $X$ and let $E_1,\dots,E_n \subset X$ denote the exceptional divisor corresponding to $x_1,\dots,x_n$, respectively. Then, for any $l_1,\dots,l_n>0$,
            \[
            \ecal L: = \pi^*\ecal M \tens \bigotimes _{i=1}^n\ecal O_Y(l_iE_i) 
            \]
            is neither ample nor anti-ample, but $\otimes$-generating. 
        \end{prop}
        \begin{proof}
       
        First of all, $\ecal L$ is neither ample nor anti-ample since $(\ecal{L}.C)< 0$ for any curve $C \subset E_i$ for some $i$ while $(\ecal{L}.D) > 0$ for any curve $D\subset Y$ with $D\cap E_i = \emp$ for all $i$. Next, since $X$ is quasi-projective and $\ecal{M}$ is ample, there are an integer $k \gg 0$ and a section $s \in \Gamma(X, \ecal{M}^{\otimes k})$ such that $X_s$ is an affine open neighborhood of $x_1, \dots , x_n$ in $X$. For each $i$, 
        let $t_i \in \ecal{O}_Y(E_i)$ be the tautological section cutting out $E_i$. Then the section $t := \pi^*s \otimes \bigotimes_{i} t_i ^{k \cdot l_i} \in \Gamma(Y, (\pi^*\ecal{M} \otimes \bigotimes \ecal{O}_Y(l_i E_i))^{\otimes k})$ satisfies $Y_t = X_s \setminus \{x_1, \dots , x_n\}$ is quasi-affine. Furthermore, the restriction of $\ecal{M}$ to $Y \setminus Y_t = (X \setminus X_s) \coprod \big(\coprod _i E_i\big)$ to each factor is either ample or anti-ample, hence it follows from Lemma \ref{lemma-usefulcriterion} that $\pi^*\ecal{M} \otimes \bigotimes \ecal{O}_Y(l_i E_i)$ is $\otimes$-generating.
        \end{proof}

\begin{theorem}
        With notations as above, suppose $X$ is smooth and $\omega_X$ is ample. Then, $\omega_Y$
        is neither ample nor anti-ample, but $\otimes$-generating. In particular, any blow-up of a smooth quasi-projective variety with ample canonical bundle 
        at finitely many points can be recovered from its perfect derived category.      
\end{theorem}

\begin{proof}
    This follows from Proposition \ref{prop-blowuptensamp} since $\omega_Y = \pi^*\omega_X \otimes \bigotimes_{i = 1}^n \ecal{O}_Y((\operatorname{dim}X-1)E_i)$.
\end{proof}

Note that the hypotheses hold for $X$ a smooth projective variety with ample canonical bundle, but also for $X$ a smooth quasi-affine variety.

\subsection{Blowups of threefolds in curves}

We use our results on ruled surfaces, Example \ref{example-ruledsurfaces}, to study $\otimes$-generation on the blowup of of a smooth projective threefold in a smooth curve. 

\begin{lemma}
\label{lemma-threefoldblowupincurve}
    Let $Y$ a smooth projective threefold over the complex numbers. Let $C \subset Y$ be a smooth integral curve. Let $b : X \to Y$ be the blowup of $Y$ in $C$ and $E \subset X$ the exceptional divisor. Denote $p :E \cong  \mathbb{P}(\ecal{I}_C/\ecal{I}_C^2) \to C$ the structure morphism and $\ecal{O}_p(1) = \ecal{O}_E(-E)$. Let $\ecal{L}$ be an ample line bundle on $X$. Then the line bundle $b^*\ecal{L} \otimes \ecal{O}_X(E)$ is $\otimes$-generating if and only if the following equivalent conditions hold:
    \begin{enumerate}
    \item $p^*(\ecal{L}|_{C}) \otimes \ecal{O}_p(-1)$ is $\otimes$-generating on $E$.
    \item The vector bundle $\ecal{I}_C/\ecal{I}_C^2$ on $C$ is unstable, and if $e$ is its degree and $d$ is the degree of the maximal destabilizing quotient line bundle, then
    \[
    \operatorname{deg}(\ecal{L}|_{C}) \neq d \text{ and }\operatorname{deg}(\ecal{L}|_{C}) < e - d. \qedhere
    \]
    \end{enumerate}
\end{lemma}

\begin{proof}
    The equivalence of (i) and (ii) is proven in Example \ref{example-ruledsurfaces}. The conditions are necessary because if $b^*\ecal{L} \otimes \ecal{O}(E)$ is $\otimes$-generating then so is its restriction to $E$, and this is $p^*(\ecal{L}|_{C}) \otimes \ecal{O}_p(-1)$, see Lemma \ref{lemma-qaffinepullback}. Conversely, assume (i) and (ii) hold. Then for any $x \in X \setminus E = Y \setminus C$ there is a global section $s$ of $\ecal{L}^{\otimes n}$ for $n \gg 0$ such that $x \in Y_s \subset Y$ is affine. Then $b^*s \otimes t$ where $t  \in \Gamma(X, \ecal{O}(n\cdot E))$ is the tautological section satisfies
    $$
    X_{b^*s \otimes t} = Y_s \setminus  C
    $$
    is quasi-affine. By Lemma \ref{lemma-usefulcriterion}, we conclude that $b^*\ecal{L} \otimes \ecal{O}(E)$ is $\otimes$-generating. 
\end{proof}

\begin{lemma}
\label{lemma-threefoldblowupincurvecanonical}
    Let $Y$ be a smooth projective threefold over the complex numbers. Assume $\omega_Y$ is ample. Let $C \subset Y$ be a smooth projective geometrically connected curve over $k$ such that the equivalent conditions of Lemma \ref{lemma-threefoldblowupincurve} hold with $\ecal{L} = \omega_Y$. Let $b: X \to Y$ be the blowup of $Y$ in $C$. Then $\omega_X$ is $\otimes$-generating. In particular, $X$ can be recovered from its perfect derived category.
\end{lemma}
\begin{proof}
    This is a direct consequence of Lemma \ref{lemma-threefoldblowupincurve} since if $E \subset X$ denotes the exceptional divisor, then $\omega_X \cong b^*\omega_Y \otimes \ecal{O}_X(E)$.
\end{proof}

\begin{example}[Special lines on a hypersurface]
    Let $Y \subset \mathbb{P}^4_{\mathbb{C}}$ be a smooth hypersurface of degree $r > 5$ containing a line $\ell \subset \mathbb{P}_{\mathbb{C}}^4$ (a linear subspace of dimension 1).  There is a short exact sequence 
    $$
    0 \to \ecal{N}_{\ell/Y} \to \ecal{N}_{\ell/\mathbb{P}^4_{\mathbb{C}}} \to \ecal{N}_{Y/\mathbb{P}^4_{\mathbb{C}}}|_{\ell} \to 0
    $$
    of vector bundles on $\ell$. The middle term is isomorphic to $\ecal{O}_\ell(1)^{\oplus 3}$ and the right term is isomorphic to $\ecal{O}_\ell(r)$. Since $\ell \cong \mathbb{P}^1_{\mathbb{C}}$ we see that $\ecal{N}_{\ell/Y} \cong \ecal{O}_\ell(a)\oplus \ecal{O}_\ell (b)$ where $a + b + r = 3$ and $a \leq  b \leq 1$. 

    \underline{Assumptions:} Assume $b =0$ or $1$, so $a  = 3- r$ or $2-r$. This is equivalent to $H^0(\ell, \ecal{N}_{\ell/Y}) \neq 0$. 

    Denote $\ecal{C} = \ecal{N}_{\ell / Y}^\vee$ the conormal bundle of $\ell \subset Y$. 
    Under our assumptions, $\ecal{C}$ is unstable of degree $e = r-3$ and with maximal destabilizing quotient line bundle $\ecal{O}(-b)$ of degree $d = -b$. We also have $\operatorname{deg}(\omega_Y)|_{\ell} = r-5$. Thus we have
    $$
    \operatorname{deg}(\omega_Y|_{\ell})> 0 \geq d \text{ and } \operatorname{deg}(\omega_Y|_{\ell}) = r-5 < (r-3) - (-b)  = e - d.
    $$
    Therefore, by Lemma \ref{lemma-threefoldblowupincurvecanonical}, the blowup of $Y$ in $\ell$ has $\otimes$-generating canonical bundle. 

    Finally, we note that such $\ell \subset Y \subset \mathbb{P}^4_{\mathbb{C}}$ exist. For example, 
    one may take $Y$ to be the Fermat hypersurface
    $$
     \{X_0^r + \cdots +X_4^r = 0\} \subset \mathbb{P}^4_{\mathbb{C}}
    $$
    and then there exist lines $\ell \subset Y$ with normal bundle $\ecal{O}_\ell(1) \oplus \ecal{O}_\ell(2-r)$, see \cite[Proof of Proposition 7.2]{larson2017normalbundleslineshypersurfaces}.
\end{example}

\subsection{Some schemes with \texorpdfstring{$\otimes$-}{tensor-}generating structure sheaf}
\label{subsection-tensampstructuresheaf}

By Lemma \ref{lemma-nefandtensorample}, if $X$ is a proper variety and $\ecal{L}$ is a globally generated line bundle on $X$ which is $\otimes$-generating, then $\ecal{L}$ is ample. This is not true when $X$ is not proper. 

\begin{lemma}[de Jong]
\label{lemma-dejong}
    There exists a quasi-projective variety $X$ over a field $k$ such that $\ecal{O}_X$ is $\otimes$-generating but not ample.
\end{lemma}

\begin{proof}
    Let $B = Bl_{(0,0)}\mathbb{A}^2_k$ and let $X = B \setminus \{p\}$ where $p$ is a $k$-point of the exceptional divisor. Then the natural map 
$$k[x,y] = H^0(\mathbb{A}^2_k, \ecal{O}_{\mathbb{A}^2_k}) \to H^0(X, \ecal{O}_X)$$
is an isomorphism but $X \to \mathbb{A}^2_k$ is not an open immersion, so $X$ is not quasi-affine \cite[\href{https://stacks.math.columbia.edu/tag/01P9}{Tag 01P9}]{stacks-project}, hence $\ecal{O}_X$ is not ample. 

The coordinates $x, y \in H^0(X, \ecal{O}_X)$ satisfy 
$$
X_x = \operatorname{Spec}(k[x^{\pm 1}, y]), \hspace{5 em}
X_y = \operatorname{Spec}(k[x, y^{\pm 1}])
$$
are affine, and $V(x,y) \cong \mathbb{A}^1_k$ so $\ecal{O}_X|_{V(x,y)}$ is ample. By Lemma \ref{lemma-usefulcriterion}, $\ecal{O}_X$ is $\otimes$-generating. 
\end{proof}

Note that in the example above, $X \to \operatorname{Spec}(H^0(X, \ecal{O}_X))$ is birational, as must be the case by Corollary \ref{corollary-steinisbirational}. It also provides a negative answer to a recent conjecture of Matsukawa \cite{matsukawa2025proper}*{Conjecture 5.5.}. As Matsukawa points out after stating the conjecture, there exist non-separated schemes with $\otimes$-generating structure sheaf. Let us provide a general construction first. 
\begin{lemma}
\label{lemma-nonseparatedgluing}
    Let $X$ be a noetherian scheme and let $\ecal L$ be a $\otimes$-generating line bundle on $X$. Take $n_1,\dots,n_k \in \bb Z$ and sections $s_i \in \Gamma(X, \ecal L^{\tens n_i})$ for $1\leq i \leq k$. Let $\pi: Y:= X\sqcup_{\cup X_{s_i}}X \to X$ denote the natural projection from the scheme $Y$ obtained by gluing two copies of $X$ along the open subscheme $\bigcup_{i=1}^k X_{s_i}$. Then, $\pi^*\ecal L$ is $\otimes$-generating.  
\end{lemma}
\begin{proof}
     Consider a family of the pullback of sections $\pi^*s_i \in \Gamma(Y,\pi^*\ecal L)$. Then, $\pi: Y_{\pi^*s_i} = \pi\inv(X_{s_i}) \to X_{s_i}$ is an isomorphism for each $i = 1,\dots,k$ and $\pi:Y\setminus \cup Y_{\pi^*s_i} = \pi\inv(X\setminus \cup X_{s_i}) \to X\setminus \cup X_{s_i}$ is a trivial double cover. Hence, since $\ecal L$ is $\otimes$-generating and its restriction to an open/closed subscheme is $\otimes$-generating, we see that $\pi^*\ecal L$ is $\otimes$-generating on each $Y_{\pi^*s_i}$ and $Y \setminus \cup Y_{\pi^*s_i}$, where the latter follows as the restriction to each connected component is $\otimes$-generating. Now, the claim follows from Lemma \ref{lemma-usefulcriterion}.  
\end{proof}
If $X$ is separated and $Y$ is integral above, then $\pi$ is a separator (see \cite{separator}*{Definition 2.3}) and hence we can also deduce that $\pi^*\ecal L$ is $\otimes$-generating from a result of Jatoba \cite[Proposition 3.7]{separator}. As a special case, we have the following. 

\begin{corollary}
    If $Y$ is affine $n$-space with doubled origin over a ring, then $\ecal{O}_Y$ is $\otimes$-generating. In particular, a scheme with a $\otimes$-generating line bundle need not be separated, have affine diagonal, have an ample line bundle, or satisfy the resolution property.
\end{corollary}

\begin{proof}
    The structure sheaf on $X = \operatorname{Spec}(R[T_1, \dots , T_n])$ is $\otimes$-generating since $X$ is affine, and $Y$ is obtained by glueing $X$ to itself along $\bigcup_{i=1}^n X_{T_i}$. Therefore, Lemma \ref{lemma-nonseparatedgluing} applies.

     It is well known that $Y$ is not separated and the diagonal of $Y$ is not affine if $n \geq 2$. A non-separated scheme never has an ample line bundle. A quasi-compact and quasi-separated scheme with the resolution property always has affine diagonal by a result of Totaro \cite[Proposition 1.3]{totaro_resolution} or \cite[\href{https://stacks.math.columbia.edu/tag/0F8C}{Tag 0F8C}]{stacks-project}.
\end{proof}

\subsection{A \texorpdfstring{$\otimes$-}{tensor-}generating line bundle on a smooth proper non-projective variety}

There exist smooth, proper, non-projective varieties $X$ with $\otimes$-generating line bundles. In fact, Hironaka's example of such a variety admits a $\otimes$-generating line bundle. We will use the following lemma to prove this.

\begin{lemma}
\label{lemma-ampleincodim2}
    Let $X$ be a normal Noetherian scheme. Let $Z \subset X$ be a closed subset of codimension $\geq 2$. Let $\ecal{L}$ be a line bundle on $X$. Assume $\ecal{L}|_{X \setminus Z}$ is ample and $\ecal{L}|_{Z}$ is $\otimes$-generating. Then $\ecal{L}$ is $\otimes$-generating.
\end{lemma}

\begin{proof}
    Let $x \in X \setminus Z$. 

    Claim: There is a closed subset $W \subset X$ such that $Z \subset W$, $x \notin W$, and $W \setminus Z$ is dense in $X$.

    To prove the claim, choose an affine open neighborhood $U$ of $x$ such that $U \subset X \setminus Z$. Set $W = X \setminus U$. This works because $W$ has pure codimension 1 in $X$ by 
 \cite[\href{https://stacks.math.columbia.edu/tag/0BCU}{Tag 0BCU}]{stacks-project} and so every point of $Z$ has a generalization in $W \setminus Z$. 

    Now since $\ecal{L}|_{X \setminus Z}$ is ample, there is an integer $n > 0$ and a global section 
    $$
    s \in \Gamma(X \setminus Z, \ecal{L}^{\otimes n}|_{X \setminus Z})
    $$
    such that $(X\setminus Z)_{s}$ is affine, $x \in (X \setminus Z)_{s}$, and $(X \setminus Z)_{s} \subset X \setminus W$. But then since $X$ is normal and $Z$ has codimension $\geq 2$, $s$ extends uniquely to a section $\tilde{s} \in \Gamma(X, \ecal{L}^{\otimes n})$, and $\tilde{s}$ vanishes on $Z$ since it vanishes on $W \setminus Z$. Hence $X_{\tilde{s}} = (X \setminus Z)_{s}$ is an affine open neighborhood of $x$. Since we can do this for every $x \in X \setminus Z$ and $\ecal{L}|_{Z}$ is $\otimes$-generating, we conclude that $\ecal{L}$ is $\otimes$-generating by Lemma \ref{lemma-usefulcriterion}.
\end{proof}

\begin{remark}
    The proof still works if $X$ is assumed an algebraic space: $X \setminus Z$ admits an ample line bundle so it is a scheme, and therefore every point $x \in X \setminus Z$ still has a neighborhood basis consisting of affine open subschemes.
\end{remark}

Now we recall Hironaka's example. Let $P$ be a smooth projective threefold over an algebraically closed field $k$ and $C, D \subset X$ two smooth, irreducible curves which intersect at two points $p_0, p_1$ which are nodes for the reducible curve $C \cup D$. Let $\pi : X \to P$ be the morphism which, over $P \setminus \{p_1\}$ is the composition of the blowup in $C$ followed by the blowup in the strict transform of $D$, and over $P \setminus \{p_0\}$ is the composition of the blowup in $D$ followed by the blowup $D$ followed by the blowup in the strict transform of $C$. Then $X$ is a smooth proper threefold which is not projective. 

More precisely, let $L \subset X$ (resp. $M \subset X$) denote the preimage of a point of $C \setminus \{p_0, p_1\}$ (resp. $D \setminus \{p_0, p_1\}$). Then the fibers of $f$ over $p_0, p_1$ are unions of two projective lines intersecting in a node
$$\pi^{-1}(p_i) = L_i \cup M_i,$$
and we have algebraic equivalences
\begin{equation}
\label{equn-algequivs}
 L_0+M_0 \sim_{alg} L \sim_{alg} L_1, \hspace{5 em} M_0 \sim_{alg} M \sim_{alg} L_1+M_1.
\end{equation}
Hence the $1$-cycle $L_0 + M_1$
is algebraically equivalent to zero. See \cite[Example B.3.4.1]{Har77}, \cite[\href{https://stacks.math.columbia.edu/tag/08J1}{Tag 08J1}]{stacks-project}, and \cite{HironakaExample}.

\begin{lemma}
    The smooth proper threefold $X$ obtained above admits a $\otimes$-generating line bundle. 
\end{lemma}

\begin{proof}
   The open $X \setminus \pi^{-1}(p_1)$ is projective over the quasi-projective scheme $P\setminus \{p_1\}$, so it has an ample line bundle. Since $X$ is smooth and $\pi^{-1}(p_1)$ has codimension 2, the ample line bundle extends uniquely to a line bundle $\ecal{L}$ on $X$. We check that $\ecal{L}$ is $\otimes$-generating using Lemma \ref{lemma-ampleincodim2}. The restriction to $X \setminus \pi^{-1}(p_1)$ is ample by construction. Since $\pi^{-1}(p_1) = L_1 \cup M_1$ it suffices to show $\operatorname{deg}\ecal{L}|_{L_1} \neq 0$ and $\operatorname{deg} \ecal{L}|_{M_1} \neq 0$. Applying (\ref{equn-algequivs}), we get
   $$
   \operatorname{deg}\ecal{L}|_{L_1} = (\ecal{L}\cdot L_1) = (\ecal{L}\cdot L_0) + (\ecal{L}\cdot M_0) > 0
   $$
   since $\ecal{L}|_{L_0}$ and $\ecal{L}|_{M_0}$ are ample. Also since $L_0 + M_1 \sim_{alg} 0$
   \[
   \operatorname{deg}\ecal{L}|_{M_1} = - \operatorname{deg}\ecal{L}|_{L_0} < 0. \qedhere
   \]
\end{proof}

\begin{remark}
\label{remark-hironakaspace}
    Hironaka also has an example of a connected smooth proper algebraic space of dimension 3 containing an irreducible curve $C$ which is algebraically equivalent to zero \cite[Example B.3.4.2]{Har77}. This algebraic space can admit no $\otimes$-generating line bundle since the restriction of any line bundle to $C$ has degree zero. 
\end{remark}

\section{Generating families of line bundles}
\label{section-families}

When we explained the proof of Theorem \ref{ref-theoremmain} to Johan de Jong, he suggested to us the following generalization. Again, the theorem and its proof work with the word ``scheme" replaced by ``algebraic space." 

\begin{theorem}
\label{theorem-multiplelinebundles}
    Let $X$ be a Noetherian scheme. Let $A \subset \operatorname{Pic}(X)$ be a finitely generated subgroup. Then $\langle A \rangle = D_{\operatorname{perf}}(X)$ if and only if for every integral closed subscheme $Z \subset X$ there are is an element $\ecal{L} \in A$ such that $\ecal{L}|_{Z}$ is big.
\end{theorem}

\begin{remark}
    As in Lemma \ref{lemma-equiv}, the hypothesis that $\langle A \rangle = D_{\operatorname{perf}}(X)$ is equivalent to: For every $K \in D_{\operatorname{QCoh}}(X)$ there is $\ecal{L} \in A$ such that $\operatorname{RHom}(\ecal{L}, K) \neq 0$.
\end{remark}

\begin{example}
    If $X$ is a Noetherian scheme with a (finite) ample family of line bundles $\ecal{L}_1, \dots , \ecal{L}_n$ (for example if $X$ is regular and has affine diagonal), and if $A \subset \operatorname{Pic}(X)$ is the subgroup generated by the $\ecal{L}_i$, then $\langle A \rangle = D_{\operatorname{perf}}(X)$. This can be proven directly but also follows from Theorem \ref{theorem-multiplelinebundles}. On the other hand, if $X$ is Hironaka's smooth proper algebraic space over a field with an irreducible curve $C$ algebraically equivalent to zero, then $X$ cannot admit a group of line bundles as in the Theorem because the restriction of any line bundle on $X$ to $C$ has degree zero, so no tensor power has a non-zero global section. 
\end{example}

\begin{proof}
$\implies$. If $\langle A \rangle = D_{\operatorname{perf}}(X)$ and $Z \subset X$ is an integral closed subscheme, then the restrictions of members of $A$ to $Z$ generate $D_{\operatorname{perf}}(Z)$ by Lemma \ref{lemma-qaffinepullback}. Thus, replacing $X$ with $Z$, we are reduced to showing that if $X$ is integral and Noetherian and $\langle A \rangle = D_{\operatorname{perf}}(X)$, then there is an $\ecal{L} \in A$ which is big. We will prove this by induction on the sum $\operatorname{rank}(A) + \# A_{tors}$.

Case 1: $H^0(X, \ecal{L}) = 0$ for each $\ecal{L} \in A$ such that $\ecal{L} \not \cong \ecal{O}_X$.

Then by part (i) of Corollary \ref{corollary-steinisbirational}, $\ecal{O}_X$ is big and we are done.

Case 2: $H^0(X, \ecal{L}) \neq 0$ for some $\ecal{L} \in A$ such that $\ecal{L} \not \cong \ecal{O}_X$.
    
    Let $0 \neq s \in H^0(X, \ecal{L})$ where $\ecal{L} \in A$ and $\ecal{L} \not \cong \ecal{O}_X$.
    Then $\ecal{L}|_{X_s} \cong \ecal{O}_{X_s}$ and therefore the image of $A$ in $\operatorname{Pic}(X_s)$ is a non-trivial quotient of $A$. Since $X_s \to X$ is affine, the image of $A$ in $\operatorname{Pic}(X_s)$ generates $D_{\operatorname{perf}}(X_s)$ by Lemma \ref{lemma-qaffinepullback}.  
    By our inductive hypothesis, we see that there are an element $\ecal{M} \in A$ and a section $t \in \Gamma(X_s, \ecal{M}|_{X_s})$ such that $(X_s)_t$ is non-empty and affine. Then by Deligne's formula, there are an integer $n > 0$ and a section 
    $$
    u \in \Gamma(X, \ecal{L}^{\otimes n} \otimes \ecal{M})
    $$
    such that $u |_{X_s} = t$. Then $us$ is a global section of a member of $A$ such that $X_{us} = X_s \cap X_u = (X_s)_t$ is non-empty and affine, as needed.

    $\impliedby$. By Noetherian induction, we may assume that for every closed subscheme $Y \subsetneq X$, the image of $A$ in $\operatorname{Pic}(Y)$ generates $D_{\operatorname{perf}}(Y)$. By Lemma \ref{lemma-sufficestochekirredcomponents} and the hypotheses, there are an element $\ecal{L} \in A$ and global section $s \in \Gamma(X, \ecal{L})$ such that $X_s$ is non-empty and affine. 

    Let $0 \neq K \in D_{\operatorname{QCoh}}(X)$. Assume $\operatorname{RHom}_X(\ecal{M}, K) = 0$ for every $\ecal{M} \in A$. In particular this holds for $\ecal{M} = \ecal{L}^{\otimes n}$ and so Deligne's formula shows that $K|_{X_s} = 0$. Then arguing as in the proof of Lemma \ref{lemma-supportedcat}, there exists an integer $r > 0$ such that if $P_r = (\ecal{L}^{\otimes -r} \xrightarrow{s^r} \ecal{O}_X)$, then $P_r^\vee \otimes^{\mathbb{L}}K \neq 0$, and $P_r^\vee \otimes^{\mathbb{L}}K = Ri_*(K')$ where $i : Y \to X$ is the inclusion of a closed subscheme supported on $V(s)$. But then by hypothesis there is a member $\ecal{M} \in A$ such that 
    $$
    0 \neq \operatorname{RHom}_Y(\ecal{M}|_Y, K') = \operatorname{RHom}_X(\ecal{M}, P_r^\vee \otimes^{\mathbb{L}}K) = \operatorname{RHom}_X(\ecal{M} \otimes^{\mathbb{L}} P_r, K).
    $$
    Since $\ecal{M} \otimes^{\mathbb{L}} P_r$ is represented by the complex $(\ecal{M} \otimes\ecal{L}^{\otimes - r}\to \ecal{M})$ we see that either 
    $
    \operatorname{RHom}(\ecal{M} \otimes\ecal{L}^{\otimes - r}, K) \neq 0
    $
    or $\operatorname{RHom}(\ecal{M}, K) \neq 0$, 
    a contradiction. 
\end{proof}

\begin{corollary}
    Let $X$ be a Noetherian scheme. Let $A \subset \operatorname{Pic}(X)$ be a finitely generated subgroup. Let $0 \neq n \in \mathbb{Z}$. Let $f : Y \to X$ be a finite, surjective morphism of schemes. Then:
    \begin{enumerate}
        \item $\langle A \rangle = D_{\operatorname{perf}}(X)$ if and only if $\langle n A \rangle = D_{\operatorname{perf}}(X)$.
        \item $\langle A \rangle = D_{\operatorname{perf}}(X)$ if and only if $\langle f^*A \rangle = D_{\operatorname{perf}}(Y)$.
        \item $\langle A \rangle = D_{\operatorname{perf}}(X)$ if and only if for every irreducible component $Z \subset X$ with the reduced subscheme structure we have $\langle A|_{Z} \rangle = D_{\operatorname{perf}}(Z)$. \qedhere
    \end{enumerate} 
\end{corollary}

\begin{proof}
    This follows from the corresponding properties of bigness as in Section \ref{subsection-someconsequences}.
\end{proof}

\medskip\noindent\textbf{Acknowledgements.} We are grateful to Johan de Jong for suggesting the generalization Theorem \ref{thm-familyintroversion} of our main result, to Martin Olsson for asking about the cone of $\otimes$-generating divisors. We would also like to thank Alexander Kuznetsov for helpful comments and Raymond Cheng for suggesting the title. We thank Leo Mayer for asking us the question that led to Example \ref{example-leo}. DI would also like to thank Riku Kurama, Ruoxi Li, Amal Mattoo and Saket Shah for interesting discussions in an early stage of the project. NO is partially supported by the National Science Foundation under Award No.
2402087.

\bibliography{bib}

\begin{bibdiv}
\begin{biblist}

\bib{bondal2002generators}{article}{
      author={Bondal, Alexei},
      author={Bergh, Michel Van~den},
       title={Generators and representability of functors in commutative and noncommutative geometry},
        date={2002},
     journal={arXiv preprint math},
        ISSN={0204218/},
}

\bib{bauer2009primer}{inproceedings}{
      author={Bauer, Thomas},
      author={Di~Rocco, Sandra},
      author={Harbourne, Brian},
      author={Kapustka, Michal},
      author={Knutsen, Andreas},
      author={Syzdek, Wioletta},
      author={Szemberg, Tomasz},
       title={A primer on {S}eshadri constants},
organization={American Mathematical Society (AMS)},
        date={2009},
   booktitle={Conference on {I}nteractions of {C}lassical and {N}umerical {A}lgebraic {G}eometry held in honor of {A}ndrew {J} {S}ommese, {U}niv {N}otre {D}ame, {N}otre {D}ame, {IN}, {M}ay 22-24, 2008},
      volume={496},
       pages={33\ndash 70},
}

\bib{biran1999constructing}{article}{
      author={Biran, Paul},
       title={Constructing new ample divisors out of old ones},
        date={1999},
     journal={Duke Math. J.},
      volume={100},
      number={1},
       pages={113\ndash 135},
}

\bib{bondal_orlov_2001}{article}{
      author={Bondal, Alexei},
      author={Orlov, Dmitri},
       title={Reconstruction of a variety from the derived category and groups of autoequivalences},
        date={2001},
     journal={Compositio Mathematica},
      volume={125},
      number={3},
       pages={327\ndash 344},
}

\bib{BPtoric}{article}{
      author={Broomhead, Nathan},
      author={Ploog, David},
       title={Autoequivalences of toric surfaces},
        date={2014},
        ISSN={00029939, 10886826},
     journal={Proceedings of the American Mathematical Society},
      volume={142},
      number={4},
       pages={1133\ndash 1146},
         url={http://www.jstor.org/stable/23808304},
}

\bib{ChrisBrav}{misc}{
      author={Brav, Chris},
       title={Generating the derived category with line bundles},
         how={MathOverflow},
        date={2010},
         url={https://mathoverflow.net/q/35863},
        note={URL:https://mathoverflow.net/q/35863 (version: 2010-08-17)},
}

\bib{ben2010integral}{article}{
      author={Ben-Zvi, David},
      author={Francis, John},
      author={Nadler, David},
       title={Integral transforms and {D}rinfeld centers in derived algebraic geometry},
        date={2010},
     journal={Journal of the American Mathematical Society},
      volume={23},
      number={4},
       pages={909\ndash 966},
}

\bib{VariationsonNagatasconjecture}{incollection}{
      author={Ciliberto, Ciro},
      author={Harbourne, Brian},
      author={Miranda, Rick},
      author={Ro\'e, Joaquim},
       title={Variations of {N}agata's conjecture},
        date={2013},
   booktitle={A celebration of algebraic geometry},
      series={Clay Math. Proc.},
      volume={18},
   publisher={Amer. Math. Soc., Providence, RI},
       pages={185\ndash 203},
      review={\MR{3114941}},
}

\bib{cox2011toric}{book}{
      author={Cox, David~A},
      author={Little, John~B},
      author={Schenck, Henry~K},
       title={Toric varieties},
   publisher={American Mathematical Soc.},
        date={2011},
      volume={124},
}

\bib{fujino2023nakai}{article}{
      author={Fujino, Osamu},
      author={Miyamoto, Keisuke},
       title={Nakai--{M}oishezon ampleness criterion for real line bundles},
        date={2023},
     journal={Mathematische Annalen},
      volume={385},
      number={1},
       pages={459\ndash 470},
}

\bib{fujino2012fundamental}{article}{
      author={Fujino, Osamu},
       title={Fundamental theorems for semi log canonical pairs},
        date={2012},
     journal={arXiv preprint arXiv:1202.5365},
}

\bib{fulger2011cones}{article}{
      author={Fulger, Mihai},
       title={The cones of effective cycles on projective bundles over curves},
        date={2011},
     journal={Mathematische Zeitschrift},
      volume={269},
      number={1},
       pages={449\ndash 459},
}

\bib{EGA4Part3}{article}{
      author={Grothendieck, Alexander},
       title={{\'E}l\'ements de g\'eom\'etrie alg\'ebrique : {IV.} \'{E}tude locale des sch\'emas et des morphismes de sch\'emas, {Troisi\`eme} partie},
    language={fr},
        date={1966},
     journal={Publications Math\'ematiques de l'IH\'ES},
      volume={28},
       pages={5\ndash 255},
         url={https://www.numdam.org/item/PMIHES_1966__28__5_0/},
}

\bib{harbourne2003Seshconst}{article}{
      author={Harbourne, Brian},
       title={Seshadri constants and very ample divisors on algebraic surfaces},
        date={2003},
     journal={Journal für die reine und angewandte Mathematik},
      volume={2003},
      number={559},
       pages={115\ndash 122},
         url={https://doi.org/10.1515/crll.2003.044},
}

\bib{Har77}{book}{
      author={Hartshorne, Robin},
       title={Algebraic {Geometry}},
      series={Graduate Texts in Mathematics},
   publisher={Springer},
     address={New York},
        date={1977},
      volume={52},
}

\bib{HironakaExample}{book}{
      author={Hironaka, Heisuke},
       title={On the theory of birational blowing-up},
   publisher={ProQuest LLC, Ann Arbor, MI},
        date={1960},
        note={Thesis (Ph.D.)--Harvard University},
      review={\MR{2939137}},
}

\bib{Krahblowups}{article}{
      author={Hu, Xianyu},
      author={Krah, Johannes},
       title={Autoequivalences of blow-ups of minimal surfaces},
        date={2024},
     journal={Bulletin of the London Mathematical Society},
      volume={56},
      number={10},
       pages={3257\ndash 3267},
      eprint={https://londmathsoc.onlinelibrary.wiley.com/doi/pdf/10.1112/blms.13131},
         url={https://londmathsoc.onlinelibrary.wiley.com/doi/abs/10.1112/blms.13131},
}

\bib{ito2024new}{article}{
      author={Ito, Daigo},
      author={Matsui, Hiroki},
       title={A new proof of the {B}ondal-{O}rlov reconstruction using {M}atsui spectra},
        date={2024},
     journal={arXiv preprint arXiv:2405.16776},
}

\bib{ito2025polarizations}{article}{
      author={Ito, Daigo},
       title={Polarizations on a triangulated category},
        date={2025},
     journal={arXiv preprint arXiv:2502.15621},
}

\bib{A_nsings}{article}{
      author={Ishii, Akira},
      author={Uehara, Hokuto},
       title={Autoequivalences of derived categories on the minimal resolutions of {$A_n$}-singularities on surfaces},
        date={2005},
        ISSN={0022-040X,1945-743X},
     journal={J. Differential Geom.},
      volume={71},
      number={3},
       pages={385\ndash 435},
         url={http://projecteuclid.org/euclid.jdg/1143571989},
      review={\MR{2198807}},
}

\bib{separator}{article}{
      author={Jatoba, V.~B.},
       title={Strong generators in {${\mathrm{D}}_{\mathrm{perf}(X)}$} for schemes with a separator},
        date={2021},
        ISSN={0002-9939,1088-6826},
     journal={Proc. Amer. Math. Soc.},
      volume={149},
      number={5},
       pages={1957\ndash 1971},
         url={https://doi.org/10.1090/proc/15353},
      review={\MR{4232189}},
}

\bib{Kawamata2002DEquivalenceAK}{article}{
      author={Kawamata, Yujiro},
       title={D-equivalence and {K}-equivalence},
        date={2002},
     journal={Journal of Differential Geometry},
      volume={61},
       pages={147\ndash 171},
}

\bib{Kollarrationalcurves}{book}{
      author={Koll\'ar, J\'anos},
       title={Rational curves on algebraic varieties},
      series={Ergebnisse der Mathematik und ihrer Grenzgebiete. 3. Folge. A Series of Modern Surveys in Mathematics [Results in Mathematics and Related Areas. 3rd Series. A Series of Modern Surveys in Mathematics]},
   publisher={Springer-Verlag, Berlin},
        date={1996},
      volume={32},
        ISBN={3-540-60168-6},
         url={https://doi.org/10.1007/978-3-662-03276-3},
      review={\MR{1440180}},
}

\bib{kuchle1996ample}{article}{
      author={K{\"u}chle, Oliver},
       title={Ample line bundles on blown up surfaces},
        date={1996},
     journal={Mathematische Annalen},
      volume={304},
      number={1},
       pages={151\ndash 155},
}

\bib{kuznetsov2011base}{article}{
      author={Kuznetsov, Alexander},
       title={Base change for semiorthogonal decompositions},
        date={2011},
     journal={Compositio Mathematica},
      volume={147},
      number={3},
       pages={852\ndash 876},
}

\bib{larson2017normalbundleslineshypersurfaces}{misc}{
      author={Larson, Hannah},
       title={Normal bundles of lines on hypersurfaces},
        date={2017},
         url={https://arxiv.org/abs/1705.01972},
}

\bib{lazarsfeld2017positivity}{book}{
      author={Lazarsfeld, Robert~K},
       title={Positivity in algebraic geometry {I}: Classical setting: line bundles and linear series},
   publisher={Springer},
        date={2017},
      volume={48},
}

\bib{Blowupofpointexample}{misc}{
      author={Libli},
       title={Generating the derived category with line bundles},
         how={MathOverflow},
         url={https://mathoverflow.net/q/259385},
        note={URL:https://mathoverflow.net/q/259385 (version: 2017-01-14)},
}

\bib{approx}{article}{
      author={Lipman, Joseph},
      author={Neeman, Amnon},
       title={Quasi-perfect scheme-maps and boundedness of the twisted inverse image functor},
        date={2007},
        ISSN={0019-2082,1945-6581},
     journal={Illinois J. Math.},
      volume={51},
      number={1},
       pages={209\ndash 236},
         url={http://projecteuclid.org/euclid.ijm/1258735333},
      review={\MR{2346195}},
}

\bib{matsukawa2025proper}{article}{
      author={Matsukawa, Hisato},
       title={Proper {F}ourier-{M}ukai partners of abelian varieties and points outside the {F}ourier-{M}ukai loci in {M}atsui spectra},
        date={2025},
     journal={arXiv preprint arXiv:2506.13527},
}

\bib{miyaoka1987chern}{incollection}{
      author={Miyaoka, Yoichi},
       title={The {C}hern classes and {K}odaira dimension of a minimal variety},
        date={1987},
   booktitle={Algebraic geometry, sendai, 1985},
      volume={10},
   publisher={Mathematical Society of Japan},
       pages={449\ndash 477},
}

\bib{nagata195914}{article}{
      author={Nagata, Masayoshi},
       title={On the 14-th problem of {H}ilbert},
        date={1959},
     journal={American Journal of Mathematics},
      volume={81},
      number={3},
       pages={766\ndash 772},
}

\bib{Neeman1992}{article}{
      author={Neeman, Amnon},
       title={The connection between the {K}-theory localization theorem of {T}homason, {T}robaugh and {Y}ao and the smashing subcategories of {B}ousfield and {R}avenel},
    language={eng},
        date={1992},
     journal={Annales scientifiques de l'École Normale Supérieure},
      volume={25},
      number={5},
       pages={547\ndash 566},
         url={http://eudml.org/doc/82328},
}

\bib{Neeman-Grothendieck}{article}{
      author={Neeman, Amnon},
       title={The {G}rothendieck duality theorem via {B}ousfield's techniques and {B}rown representability},
        date={1996},
     journal={J. Amer. Math. Soc.},
      volume={9},
      number={1},
       pages={205\ndash 236},
}

\bib{oda1988convex}{book}{
      author={Oda, Tadao},
       title={Convex bodies and algebraic geometry: An introduction to the theory of toric varieties},
   publisher={Springer},
        date={1988},
}

\bib{olander2021rouquierdimensionquasiaffineschemes}{article}{
      author={Olander, Noah},
       title={The {R}ouquier dimension of quasi-affine schemes},
        date={2021},
     journal={arXiv preprint arXiv:2108.12005},
}

\bib{orlov2003derived}{article}{
      author={Orlov, Dmitri},
       title={Derived categories of coherent sheaves and equivalences between them},
        date={2003},
     journal={Russian Mathematical Surveys},
      volume={58},
      number={3},
       pages={511},
}

\bib{Orlov_dimension}{article}{
      author={Orlov, Dmitri},
       title={Remarks on generators and dimensions of triangulated categories},
        date={2009},
     journal={Moscow Math. J.},
      volume={vol.9},
      number={1},
       pages={153\ndash 159},
}

\bib{rouquier2008dimensions}{article}{
      author={Rouquier, Rapha{\"e}l},
       title={Dimensions of triangulated categories},
        date={2008},
     journal={Journal of K-theory},
      volume={1},
      number={2},
       pages={193\ndash 256},
}

\bib{stacks-project}{misc}{
      author={{Stacks project authors}, The},
       title={The stacks project},
         how={\url{https://stacks.math.columbia.edu}},
        date={2025},
}

\bib{totaro_resolution}{article}{
      author={Totaro, Burt},
       title={The resolution property for schemes and stacks},
        date={2004},
        ISSN={0075-4102},
     journal={J. Reine Angew. Math.},
      volume={577},
       pages={1\ndash 22},
}

\bib{thomason2007higher}{incollection}{
      author={Thomason, Robert~Wayne},
      author={Trobaugh, Thomas},
       title={Higher algebraic {$K$}-theory of schemes and of derived categories},
        date={1990},
   booktitle={The {G}rothendieck {F}estschrift, {V}ol.\ {III}},
      series={Progr. Math.},
      volume={88},
   publisher={Birkh\"auser Boston},
     address={Boston, MA},
       pages={247\ndash 435},
}

\bib{toen2007moduli}{inproceedings}{
      author={To{\"e}n, Bertrand},
      author={Vaqui{\'e}, Michel},
       title={Moduli of objects in dg-categories},
        date={2007},
   booktitle={Annales scientifiques de l'ecole normale sup{\'e}rieure},
      volume={40},
       pages={387\ndash 444},
}

\bib{Ellipticruledsurface}{article}{
      author={Uehara, Hokuto},
       title={Fourier-{M}ukai partners of elliptic ruled surfaces},
        date={2017},
        ISSN={0002-9939,1088-6826},
     journal={Proc. Amer. Math. Soc.},
      volume={145},
      number={8},
       pages={3221\ndash 3232},
         url={https://doi.org/10.1090/proc/13471},
      review={\MR{3652778}},
}

\bib{xu1995divisors}{article}{
      author={Xu, Geng},
       title={Divisors on the blow-up of the projective plane.},
        date={1995},
     journal={manuscripta mathematica},
      volume={86},
      number={1},
}

\end{biblist}
\end{bibdiv}

\end{document}